\documentclass[a4paper,12pt,reqno]{amsart}
\usepackage{amsthm, mathrsfs,amssymb,amsmath,}
\usepackage{enumerate}
\usepackage{tikz}
\usepackage{pgfplots}
\usepackage{mathtools}
\usepackage[foot]{amsaddr}
\makeatletter
\@namedef{subjclassname@2010}{%
	\textup{2010} Mathematics Subject Classification}
\makeatother

\pgfplotsset{compat=1.18}

\usepackage{hyperref}
\allowdisplaybreaks
\numberwithin{equation}{section}

\usepackage{graphicx}
\usepackage{enumitem}

\usepackage{graphicx} 
\newtheorem{theorem}{Theorem}[section]
\newtheorem{lemma}[theorem]{Lemma}
\usepackage{xcolor}

\theoremstyle{definition}
\newtheorem{definition}[theorem]{Definition}
\newtheorem{example}[theorem]{Example}
\newtheorem{proposition}[theorem]{Proposition}

\theoremstyle{remark}

\numberwithin{equation}{section}

\definecolor{trp}{rgb}{1,1,1}

\definecolor{red}{rgb}{1,0,.2}

\definecolor{blue}{rgb}{0,0,1}

\definecolor{rgrey}{rgb}{.8,0.4,.4}  
\definecolor{grey}{rgb}{.13,.13,.13}  

\definecolor{green}{rgb}{0.0,0.4,0.2}

\numberwithin{equation}{section}

\newcommand{\R}{\mathbb{R}}
\newcommand{\N}{\mathbb{N}}

\begin{document}
	
	\title[]{Hausdorff dimension of self-similar measures and sets with common fixed point structure}
	
	
~~~~~~~
    
	\author{Bal\'azs B\'ar\'any}
	\author{Manuj Verma}
        \thanks{BB and MV acknowledge support from the grants NKFI~K142169 and NKFI KKP144059 "Fractal geometry and applications" Research Group.}
	\address{Department of Stochastics \\ Institute of Mathematics \\  Budapest University of Technology and Economics \\ M\H{u}egyetem rkp. 3., H-1111 Budapest, Hungary }
	\address[Bal\'azs B\'ar\'any]{barany.balazs@ttk.bme.hu}
	\address[Manuj Verma]{mathmanuj@gmail.com }

	
	%

	
	
	

 \date{\today}
	
	\begin{abstract}
	In this paper, we study the Hausdorff dimension of self-similar measures and sets on the real line, where the generating iterated function system consists of maps that share a common fixed point. In particular, we will show that out of a Hausdorff co-dimension one exceptional set of natural parameters, such systems satisfy a weak exponential separation. This significantly strengthens the previous result of B\'ar\'any and Szv\'ak \cite{BS2021}. As an application, we give the Hausdorff dimension of self-affine measures supported on the generalised 4-corner set.
	\end{abstract}
	
	   \maketitle
	
	
	\section{Introduction and Statements}\label{sec:intro}

For $N\geq 2$, the system $\mathcal{I}=\big\{ f_1,f_2,\ldots,f_N \big\}$ on $\mathbb{R}^d$ is called an \texttt{iterated function system (IFS)}, if each $f_i$ is a contraction map on $\mathbb{R}^d$ for $i\in \{1,2,\ldots, N\}$. Hutchinson \cite{Hutchinson} proved that there exists a unique non-empty compact set $A\subset \mathbb{R}^d$ such that 
$$A=\bigcup_{i=1}^{N}f_i(A).$$ 
Let $\mathbf{p}=(p_1,p_2,\dots,p_N)$ be a probability vector. Hutchinson \cite{Hutchinson} also proved that there exists a unique Borel probability measure $\mu$ supported on $A$ such that 
$$\mu(B)=\sum_{i=1}^{N}p_{i}\mu(f_{i}^{-1}(B))$$
for all Borel sets $B\subset \mathbb{R}^d$. The set $A$ is called the \texttt{attractor} of IFS $\mathcal{I}$, and the measure $\mu$ is known as the \texttt{stationary measure} corresponding to IFS $\mathcal{I}$ with the probability vector $\mathbf{p}$. 

    This paper focuses on the dimension theory of self-similar IFSs on the line. The IFS $\mathcal{I}$ is called a \texttt{self-similar iterated function system} on $\mathbb{R}$ if each $f_{i}:\mathbb{R}\to \mathbb{R}$ is a similarity map for $i\in\{1,2,\dots,N\}$, i.e.
$f_i(x)=r_ix+c_i$ where $0<|r_{i}|<1$ and $c_{i}\in \mathbb{R}$. Then, the attractor $A$ is called a \texttt{self-similar set} corresponding to the self-similar IFS $\mathcal{I}$. The stationary measure $\mu$ is known as the \texttt{self-similar measure} corresponding to the self-similar IFS $\mathcal{I}$ with probability vector $\mathbf{p}=(p_1,p_2,\dots,p_N)$. For later purposes, let us recall the definition of the Hausdorff dimension here.

Let $F\subseteq \mathbb{R}^d$. We say that $\{U_{i}\}_{i=1}^\infty$ is a $\delta$ cover of $F$ if $F\subset \bigcup\limits_{i=1}^{\infty}U_{i}$ and $0 <|U_{i}|\leq \delta$ for each ${i}$, where $|U_i|$ denotes the diameter of the set $U_i$. For each $\delta>0$ and $s\geq 0$, we define
		$$\mathcal{H}_{\delta}^{s}(F):=\inf\Big\{\sum_{i=1}^{\infty}|U_{i}|^{s} : \{U_{i}\}_{i=1}^{\infty}\text{ is a }\delta\text{-cover of }F\Big\}\text{ and }\mathcal{H}^{s}(F):=\lim_{\delta\to 0+}\mathcal{H}_{\delta}^{s}(F).$$
		We call $\mathcal{H}^{s}(F)$ the \texttt{$s$-dimensional Hausdorff measure} of the set $F$. Using this, the \texttt{Hausdorff dimension} of the set $F$ is defined by 
		$$\dim_{H}(F)=\inf\{s\geq 0 : \mathcal{H}^{s}(F)=0\}.$$
For the basic properties, we refer the reader to Falconer's book \cite{Falconer1990}. Another widely used dimension concept is the box-counting (or Minkowski) dimension; however, Falconer \cite{Falconer1989} showed that these dimensions coincide for self-similar sets, so we omit them.

For a Borel probability measure $\mu$ in $\mathbb{R}^d$, the \texttt{lower and upper local dimension} of the measure $\mu$ at the point $x\in \mathbb{R}^d$ defined as $$\underline{d}_{\mu}(x)=\liminf_{r\to 0}\frac{\log \mu(B(x,r))}{\log r},\quad \overline{d}_{\mu}(x)=\limsup_{r\to 0}\frac{\log \mu(B(x,r))}{\log r},$$
  where $B(x,r)$ denotes the open ball of radius $r$ with center $x$. If the limit exists at the point $x$, then we call it the local dimension and denote it by $d_\mu(x)$.
We say that the measure $\mu$ is \texttt{exact dimensional} if there is a constant $c$ such that $d_{\mu}(x)=c$ for $\mu$-almost every $x\in \mathbb{R}^d$. We will denote this value by $\dim_{H}(\mu)$. Feng and Hu \cite{FengandHu} proved that every self-similar measure $\mu$ is exact dimensional. For further properties of the dimension of measures, see Przytycki and Urba\'nski \cite[Section~7.6]{PU10}.

Let us define the \texttt{similarity dimension} for the self-similar IFS $\mathcal{I}$ to be the unique number $s_{0}\in \mathbb{R}$ such that $$\sum_{i=1}^{N}|r_{i}|^{s_0}=1.$$ Furthermore, the quantities 
$$h_{\mathbf{p}}=-\sum_{i=1}^{N}p_{i}\log p_{i}\quad \text{and}\quad \chi(\mathbf{p})=-\sum_{i=1}^{N}p_i\log |r_{i}|$$
are known respectively as the \texttt{entropy} and the \texttt{Lyapunov exponent} of the self-similar measure $\mu$ corresponding to the self-similar IFS $\mathcal{I}$ with probability vector $\mathbf{p}$. It is well known that 
$$\dim_{H}(A)\leq \min\{1,s_0\} \quad \text{and}\quad \dim_{H}(\mu)\leq \min\bigg\{1, \frac{h_{\mathbf{p}}}{\chi(\mathbf{p})}\bigg\}$$ 
regardless of any further assumptions. However, equalities appear in the above expressions under some suitable separation conditions. We say that the IFS $\mathcal{I}$ satisfies the \texttt{open set condition (OSC)}, if there is a nonempty and bounded open set $O\subset \mathbb{R}^d$ with $f_i(O) \subset O$ and $ f_i(O)\cap f_j(O)= \emptyset $ for every $i\ne j \in \{1,2,\dots, N\}$.

\begin{theorem}[Hutchinson \cite{Hutchinson}, Cawley and Mauldin \cite{CaMa}]
    Let the self-similar IFS $\mathcal{I}$ on $\mathbb{R}$. If $\mathcal{I}$ satisfies the open set condition (OSC), then
   \begin{equation}\label{eq:nodrop}\dim_{H}(A)= s_0 \quad \text{and}\quad \dim_{H}(\mu)= \frac{h_{\mathbf{p}}}{\chi(\mathbf{p})},\end{equation}
   where $A$ is the self-similar attractor of the self-similar IFS $\mathcal{I}$ and $\mu$ is the self-similar measure corresponding to the self-similar IFS $\mathcal{I}$ with probability vector $\mathbf{p}$.
\end{theorem}

Hochman \cite{Hochman2014} introduced the Exponential Separation Condition for the self-similar IFSs, which is a much weaker separation condition than the OSC. We say that the self-similar IFS $\mathcal{I}$ on $\mathbb{R}$ satisfies the \texttt{Exponential Separation Condition (ESC)} if there exists a constant $\eta>0$ such that for infinitely many values of $n\in \mathbb{N}$ for all $i_1,\ldots,i_n,j_1,\ldots,j_n\in \{1,\ldots, N\}$ with $(i_1,\ldots,i_n)\ne(j_1,\ldots,j_n)$ and $r_{i_1}\cdots r_{i_n}=r_{j_1}\cdots r_{j_n}$, we have 
   $$|f_{i_1}\circ\cdots\circ f_{i_n}(0)-f_{j_1}\circ \cdots\circ f_{j_n}(0)|\geq {\eta}^{n}.$$ 
In a proper sense, the exponential separation is a typical property for self-similar IFSs. For example, suppose the contraction ratios and translations are analytic maps of some one-dimensional parameter, and the system is not degenerate. In that case, ESC holds out of a zero-dimensional set of exceptional parameters. In particular, for fixed contraction ratios with $\max_{i\neq j}\{|r_i|+|r_j|\}<1$, the dimension of translation vectors $(c_i)_{i=1}^N\in\mathbb{R}^N$, for which $\mathcal{I}$ does not satisfy ESC, is at most $N-1$, see \cite[Corollary~6.8.12]{BSS23}\footnote{We note that the assumption $\max_{i\neq j}\{|r_i|+|r_j|\}<1$ is missing incorrectly in the statement of \cite[Corollary 6.8.12]{BSS23}}. Hochman \cite{Hochman2014} proved the following groundbreaking result for the dimension theory of self-similar IFSs.

\begin{theorem}\label{MainHochmanresult}[Hochman~\cite{Hochman2014}]
    If the self-similar IFS $\mathcal{I}$ on $\mathbb{R}$ satisfies the ESC, then 
    \begin{equation}\label{eq:nodrop2}\dim_{H}(A)=\min\{1,s_0\} \quad \text{and}\quad \dim_{H}(\mu)=\min\left\{1,\frac{h_{\mathbf{p}}}{\chi(\mathbf{p})}\right\}.\end{equation}
\end{theorem}

The ESC implies that the self-similar IFS $\mathcal{I}$ has no exact overlap, i.e. there are no finite sequences of symbols $(i_1,\ldots,i_n)\ne(j_1,\ldots,j_m)$ such that $f_{i_1}\circ\cdots\circ f_{i_n}\equiv f_{j_1}\circ \cdots\circ f_{j_m}$. When exact overlaps occur, a better bound can be obtained by counting them. Lau and Ngai \cite{LauNgai} and Zerner \cite{Zerner} introduced the Weak Separation Condition (WSC), which allows exact overlaps. Still, the not-exactly-overlapping maps must be relatively far apart, comparable to the contraction ratios of the maps. This allows us to calculate the dimension of the attractor as a limit of the similarity dimensions of higher iterates. In some special cases, when the contraction ratios are reciprocals of integers, it is possible to calculate the dimension of the corresponding self-similar measures, see Ruiz \cite{Ruiz}.

The situation becomes significantly more difficult if the overlapping structure of the self-similar IFS $\mathcal{I}$ is complicated, for instance, when some of the maps of the IFS share the same fixed point. Let $(t_i)_{i=1}^N\in\mathbb{R}^N$ be a vector with pairwise different entries. For every $i\in\{1,\ldots,N\}$, let $n_i\geq1$ be an integer. For every $j\in\{1,\ldots,n_i\}$, let $\lambda_{i,j}\in(0,1)$. Consider the following self-similar IFS on the real line:
	\begin{equation}\label{eq:mainIFS}\mathcal{S}=\{f_{i,j}(x)=\lambda_{i,j}x+t_{i}(1-\lambda_{i,j})\}_{1\leq i\leq N, 1\leq j\leq n_i}.\end{equation}
If $n_i\geq2$ for some $i$, then the maps $f_{i,j}$ and $f_{i,j'}$ with $j\neq j'$ share the same fixed point, which causes many exact overlaps, for example $f_{i,j}\circ f_{i,j'}\equiv f_{i,j'}\circ f_{i,j}$. Furthermore, if $\frac{\log\lambda_{i,j}}{\log\lambda_{i,j'}}\notin\mathbb{Q}$ then the IFS $\mathcal{S}$ does satisfy the WSC, see Fraser \cite{FrasernoWSC}, and if $\frac{\log\lambda_{i,j}}{\log\lambda_{i,j'}}=\frac{p}{q}\in\mathbb{Q}$ then it causes further exact overlaps, since composing $f_{i,j'}$ with itself $p$-times and composing $f_{i,j}$ with itself $q$-times result the same map.

 The dimension theory of such systems has been studied previously. B\'ar\'any \cite{Barany1} and with Szv\'ak in \cite{BS2021} studied the Hausdorff dimension of the attractor of the special case $\{\lambda_{1,1}x,\lambda_{1,2}x,\lambda_{2,1}x+1\}$. They showed that if $\lambda_{1,1}+\lambda_{2,1}<1$ then for Lebesgue almost every $\lambda_{1,2}\in(0,\lambda_{1,1})$ the dimension of the attractor is $\min\{1,s\}$, where $s$ is the unique root of the equation
$$\lambda_{1,1}^{s}+\lambda_{1,2}^s-\lambda_{1,1}^s\lambda_{1,2}^s+\lambda_{2,1}^s=1.$$
Moreover, in  \cite{BS2021}, B\'ar\'any and Szv\'ak gave a formula for the dimension of self-similar measures corresponding to that IFS for Lebesgue almost every parameter under the technical assumption that $0<\lambda_{1,1},\lambda_{1,2},\lambda_{2,1}<1/9$. Further generalisations can be found in \cite{Barany2, Barany4corner}.

In this paper, we extend significantly the dimension result for such systems. For simplicity, let us introduce the following notations: let $L=\sum_{i=1}^Nn_i$ be the number of maps in $\mathcal{S}$, let $\underline{\lambda}=(\lambda_{i,j})_{1\leq i\leq N,1\leq j\leq n_i}$ be the $L$-dimensional vector formed by the contraction ratios, and let $\underline{t}=(t_i)_{1\leq i\leq N}$ be the vector formed by the fixed points. To emphasize the dependence on the the parameters, we may write $\mathcal{S}_{\underline{\lambda},\underline{t}}$ for the system defined in \eqref{eq:mainIFS}, $A_{\underline{\lambda},\underline{t}}$ for the attractor of $\mathcal{S}_{\underline{\lambda},\underline{t}}$ and $\mu_{\underline{\lambda},\underline{t}}^{\mathbf{p}}$ for the self-similar measure corresponding to the IFS $\mathcal{S}_{\underline{\lambda},\underline{t}}$ and probability vector $\mathbf{p}$. Our main result on the dimension of self-similar measures is the following:

\begin{theorem}\label{thm:main} Let $\mathcal{S}_{\underline{\lambda},\underline{t}}$ be a self-similar IFS as in \eqref{eq:mainIFS}. There exists a set $\mathbf{E}\subset(0,1)^{L}$ with $\dim_H(\mathbf{E})\leq L-1$ such that for every $\underline{\lambda}\in(0,1)^{L}\setminus\mathbf{E}$ there exists a set $\mathbf{F}\subset\mathbb{R}^N$ with $\dim_H(\mathbf{F})\leq N-1$ such that for every $\underline{t}\in\mathbb{R}^N\setminus\mathbf{F}$ the following holds: for every probability vector $\mathbf{p}=(p_{i,j})_{1\leq i\leq N, 1\leq j\leq n_i}$ 
	$$\dim_{H}(\mu_{\underline{\lambda},\underline{t}}^{\mathbf{p}})=\min\bigg\{1, \frac{-\sum_{i=1}^{N}\sum_{j=1}^{n_i}p_{i,j}\log(p_{i,j})+\Phi(\mathbf{p})}{-\sum_{i=1}^{N}\sum_{j=1}^{n_i}p_{i,j}\log(\lambda_{i,j})}\bigg\},$$
		where 
        \begin{equation}\label{eq:phip}
        \Phi(\mathbf{p})=\sum_{l=1}^{N}\sum_{m=1}^{n_{l}}\sum_{k=0}^{\infty}\sum_{q=0}^{k}\binom{k}{q}(p_{l,m})^{q+1}\left(\sum_{\substack{j=1\\j\neq m}}^{n_l}p_{l,j}\right)^{k-q}\bigg(1-\sum_{j=1}^{n_l}p_{l,j}\bigg)\log\left(\frac{q+1}{k+1}\right).
        \end{equation}
\end{theorem}

This result significantly strengthens the previous result of B\'ar\'any and Szv\'ak \cite[Theorem~1.2]{BS2021}. Although the expression $\Phi(\mathbf{p})$ is tedious, it has a simple heuristic. It measures the proportion of the symbol that appears first in the first block, which is formed by the maps sharing the same fixed point. Namely, suppose $X_{1}, X_{2},\dots$ are independent and identically distributed random variables such that $\mathbb{P}(X_r=(i,j))=p_{i,j}$ for all $i\in\{1,\dots, N\}$ and $j\in \{1,2,\dots,n_i\}$. Let $k_{i,j}:=\min\{r\geq 1: X_r\notin\{(i,1),\ldots,(i,n_i)\} \}$, and let $Y_{i,j}:=\#\{1\leq r\leq k_{i,j}-1: X_r=(i,j)\}$ for all $j\in \{1,2,\dots,n_i\}.$ Then, the quantity $\Phi(\mathbf{p})$ can be written in the form 
$$\Phi(\mathbf{p})= \mathbb{E}\bigg(\log\bigg(\frac{Y_{X_1}}{k_{X_1}-1}\bigg)\bigg).$$ 

Next, we state a similar result for the dimension of the attractor of $\mathcal{S}_{\underline{\lambda},\underline{t}}$.

\begin{theorem}\label{dimattractor}
  Let $\mathcal{S}_{\underline{\lambda},\underline{t}}$ be a self-similar IFS as in \eqref{eq:mainIFS}. There exists a set $\mathbf{E}\subset(0,1)^{L}$ with $\dim_H(\mathbf{E})\leq L-1$ such that for every $\underline{\lambda}\in(0,1)^{L}\setminus\mathbf{E}$ there exists a set $\mathbf{F}\subset\mathbb{R}^N$ with $\dim_H(\mathbf{F})\leq N-1$ such that for every $\underline{t}\in\mathbb{R}^N\setminus\mathbf{F}$,
  $$\dim_H(A_{\underline{\lambda},\underline{t}})=\min\{1,s_0\},$$ where $A_{\underline{\lambda},\underline{t}}$ is the attractor of $\mathcal{S}_{\underline{\lambda},\underline{t}}$ and $s_0$ is the unique solution of the equation
  \begin{equation}\label{eq:dimatt}
  \sum_{i=1}^N\prod_{j=1}^{n_i}(1-(\lambda_{i,j})^{s_0})=N-1.
  \end{equation}
\end{theorem}

This significantly strengthens the previous results in \cite{Barany1, Barany2, Barany4corner, BS2021}.

\section{Applications and discussion}

\subsection{Typical self-similar systems} Before we prove our main result, Theorem~\ref{thm:main}, we show some applications of it. First, we show that the assumption $\max_{i\neq j}\{|r_i|+|r_j|\}<1$ in \cite[Corollary~6.8.12]{BSS23} can be strengthened.

\begin{theorem}\label{thm:main2}
    Let $\mathcal{I}=\{f_i(x)=r_ix+c_i\}_{i=1}^N$ be a self-similar IFS. Then there exists a set $\mathbf{E}\subset(0,1)^{N}$ with $\dim_H(\mathbf{E})\leq N-1$ such that for every $(r_{i})_{i=1}^N\in(0,1)^{N}\setminus\mathbf{E}$ there exists a set $\mathbf{F}\subset\mathbb{R}^N$ with $\dim_H(\mathbf{F})\leq N-1$ such that for every $(c_i)_{i=1}^N\in\mathbb{R}^N\setminus\mathbf{F}$ the following holds: for every self-similar measure $\mu$ corresponding to the self-similar IFS $\mathcal{I}$ and the probability vector $\mathbf{p}=(p_{i})_{i=1}^{N}$
		$$\dim_{H}(\mu)=\min\bigg\{1, \frac{-\sum_{i=1}^{N}p_{i}\log(p_{i})}{-\sum_{i=1}^{N}p_{i}\log(r_{i})}\bigg\}.$$
\end{theorem}

\begin{proof}
    Since $n_i=1$ for every $i=1,\ldots,N$, it is easy to see that the expression $\Phi(\mathbf{p})$ in Theorem~\ref{thm:main} equals $0$. Thus, the claim follows.
\end{proof}

\subsection{Estimating the \texorpdfstring{$\Phi(\mathbf{p})$}{Φ(p)}} Now, we will establish a bound for $\Phi(\mathbf{p})$, which might help estimate from below the Hausdorff dimension of self-similar measures with common fixed point structure.

\begin{proposition}\label{lem:lowerbound} Let $\mathbf{p}=(p_{i,j})_{1\leq i\leq N, 1\leq j\leq n_i}$ be a probability vector and let $\Phi(\mathbf{p})$ be defined as in \eqref{eq:phip}. Then
	$$\Phi(\mathbf{p}) \geq \sum_{l=1}^{N}\sum_{m=1}^{n_{l}}p_{l,{m}} \log \bigg(p_{l,{m}}+\sum_{\substack{i\ne l\\ i=1}}^{N}\sum_{j=1}^{n_i}p_{i,j}\bigg).$$
\end{proposition}
\begin{proof}
Let us recall the representation of $\Phi(\mathbf{p})$ from Section~\ref{sec:intro}. Suppose $X_{1},X_{2},\dots$ are independent and identically distributed (i.i.d.) random variables such that $\mathbb{P}(X_r=(i,j))=p_{i,j}$ for all  $i\in\{1,\dots,N\}$ and $j\in \{1,2,\dots,n_i\}$. Let $k_{i,j}:=\min\{r\geq 1: X_r\notin\{(i,1),\ldots,(i,n_i)\} \},$ and let $Y_{i,j}:=\#\{1\leq r\leq k_{i,j}-1: X_r=(i,j)\}$ for all $j\in \{1,2,\dots,n_i\}.$ Then, the quantity $\Phi(\mathbf{p})$ can be written in the following form 
	$$ \Phi(\mathbf{p})= \mathbb{E}\bigg(\log\bigg(\frac{Y_{X_1}}{k_{X_1}-1}\bigg)\bigg)=\sum_{l=1}^{N}\sum_{m=1}^{n_{l}}p_{l,{m}} \mathbb{E}\bigg(\log\bigg(\frac{Y_{l,m}}{k_{l,m}-1}\bigg) \bigg| X_{1}=(l,m)\bigg),$$ where
	\begin{align*}
	&\mathbb{E}\bigg(\log\bigg(\frac{Y_{l,m}}{k_{l,m}-1}\bigg) \bigg| X_{1}=(l,m)\bigg)\\
    &\quad\quad= \sum_{k=0}^{\infty}\sum_{q=0}^{k}\binom{k}{q}p_{l,m}^q\left(\sum_{\substack{j=1\\j\neq m}}^{n_l}p_{l,j}\right)^{k-q}\bigg(1-\sum_{j=1}^{n_l}p_{l,j}\bigg)\log\bigg(\frac{q+1}{k+1}\bigg).
	\end{align*}
Clearly,
\begin{align*}
\mathbb{E}\bigg(\log\bigg(\frac{Y_{l,m}}{k_{l,m}-1}\bigg) \bigg| X_{1}=(l,m)\bigg)&=\mathbb{E}\bigg(-\log\bigg(\frac{k_{l,m}-1}{Y_{l,m}}\bigg) \bigg| X_{1}=(l,m)\bigg)\\
&\geq -\log \mathbb{E}\bigg(\frac{k_{l,m}-1}{Y_{l,m}} \bigg| X_{1}=(l,m)\bigg),
\end{align*}
where we used Jensen's inequality. Then
\begin{align*}
\mathbb{E}\bigg(\frac{k_{l,m}-1}{Y_{X_r}}| X_{1}=(l,m)\bigg)&=\sum_{k=0}^{\infty}\sum_{q=0}^{k}\binom{k}{q}p_{l,m}^q\left(\sum_{\substack{j=1\\j\neq m}}^{n_l}p_{l,j}\right)^{k-q}\bigg(1-\sum_{j=1}^{n_l}p_{l,j}\bigg)\frac{k+1}{q+1}\\
&=\sum_{k=0}^{\infty}\sum_{q=0}^{k}\binom{k+1}{q+1}p_{l,m}^{q}\left(\sum_{\substack{j=1\\j\neq m}}^{n_l}p_{l,j}\right)^{k-q}\bigg(1-\sum_{j=1}^{n_l}p_{l,j}\bigg)\\
&=\sum_{k=1}^{\infty}\bigg(1-\sum_{j=1}^{n_l}p_{l,j}\bigg)\frac{1}{p_{l,m}}\left(\bigg(\sum_{j=1}^{n_l}p_{l_j}\bigg)^{k}-\bigg(\sum_{\substack{j=1\\j\ne m}}^{n_l}p_{l_j}\bigg)^{k}\right)\\
&=\bigg(1-\sum_{j=1}^{n_l}p_{l_j}\bigg)\frac{1}{p_{l,m}}\bigg(\frac{\sum_{j=1}^{n_l}p_{l,j}}{1-\sum_{j=1}^{n_l}p_{l,j}}-\frac{\sum_{j=1,j\ne m}^{n_l}p_{l,j}}{1-\sum_{j=1,j\ne m}^{n_l}p_{l,j}}\bigg)\\
&=\left(1-\sum_{\substack{j=1\\j\ne m}}^{n_l}p_{l,j}\right)^{-1}=\left(p_{l,{m}}+\sum_{\substack{i\ne l\\ i=1}}^{N}\sum_{j=1}^{n_i}p_{i,j}\right)^{-1}
\end{align*}
Thus, from the above, we have 
$$\mathbb{E}\bigg(\log\bigg(\frac{Y_{l,m}}{k_{l,m}-1}\bigg) \bigg| X_{1}=(l,m))\bigg)\geq \log \bigg(p_{l,{m}}+\sum_{\substack{i\ne l\\ i=1}}^{N}\sum_{j=1}^{n_i}p_{i,j}\bigg).$$ This implies that 
$$\Phi(\mathbf{p}) \geq \sum_{l=1}^{N}\sum_{m=1}^{n_{l}}p_{l,{m}} \log \bigg(p_{l,{m}}+\sum_{\substack{i\ne l\\ i=1}}^{N}\sum_{j=1}^{n_i}p_{i,j}\bigg).$$
This completes the proof.
\end{proof}

\subsection{Generalised 4-corner set} Using Theorem~\ref{thm:main}, we can also calculate the Hausdorff dimension of self-affine measures supported on the generalised $4$-corner set. Let 
\begin{equation}\label{eq:IFS4corner}
\begin{array}{cc}
F_{1}(x,y)=(\gamma_{1,1}x,\lambda_{1,1}y), &F _{2}(x,y)=(\gamma_{1,2}x,\lambda_{2,1}y+(1-\lambda_{2,1})),\\ 
F_{3}(x,y)=(\gamma_{2,1}x+(1-\gamma_{2,1}),\lambda_{1,2}y), &F_{4}(x,y)=(\gamma_{2,2}x+(1-\gamma_{2,2}),\lambda_{2,2}y+(1-\lambda_{2,2}))
\end{array}
\end{equation}
be a self-affine IFS on $\mathbb{R}^2$ such that 
\begin{equation}\label{eq:cond4corner}
\left.\begin{array}{c}
0<\gamma_{i,j},\lambda_{i,j}\text{ for }i=1,2, j=1,2,\\
\gamma_{1,1}+\gamma_{2,1}\leq 1,\ \gamma_{1,2}+\gamma_{2,2}\leq 1,\ \lambda_{1,1}+\lambda_{2,1}\leq1,\ \lambda_{1,2}+\lambda_{2,2}\leq1,\\
\min\{\gamma_{1,2}+\gamma_{2,1},\lambda_{1,2}+\lambda_{2,1}\}\leq1\text{ and }\min\{\gamma_{1,1}+\gamma_{2,2},\lambda_{1,1}+\lambda_{2,2}\}\leq1.
\end{array}\right\}
\end{equation}
The conditions in \eqref{eq:cond4corner} guarantee that the IFS $\{F_1,F_2,F_3,F_4\}$ satisfies the rectangular open set condition. In particular, $F_i((0,1)^2)\subset(0,1)^2$ and $F_i((0,1)^2)\cap F_j((0,1)^2)=\emptyset$ for every $i\neq j\in\{1,2,3,4\}$. We call the attractor $C=\bigcup_{i=1}^4F_i(C)$ of $\{F_1,\ldots,F_4\}$ the generalized 4-corner set. In \cite{Barany4corner}, the first author studied the box-counting dimension of the 4-corner set for Lebesgue typical parameters. Now, we determine the Hausdorff dimension of self-affine measures supported on $C$ for typical parameters in a certain sense, and the Hausdorff dimension of the generalised 4-corner set in a region of parameters.

Let $\mathbf{p}=(p_1,\ldots,p_4)$ be a probability vector, and let
\begin{align*}
\chi_x(\mathbf{p})&=-p_1\log\gamma_{1,1}-p_2\log\gamma_{1,2}-p_3\log\gamma_{2,1}-p_4\log\gamma_{2,2},\\
\chi_y(\mathbf{p})&=-p_1\log\lambda_{1,1}-p_2\log\lambda_{2,1}-p_3\log\lambda_{1,2}-p_4\log\lambda_{2,2}.
\end{align*}

\begin{theorem}\label{thm:4corner} Let $\{F_1,F_2,F_3,F_4\}$ be a self-affine IFS as in \eqref{eq:IFS4corner}. There exist sets $\mathbf{E}_x,\mathbf{E}_y\subset(0,1)^{4}$ with $\dim_H(\mathbf{E}_x),\dim_H(\mathbf{E}_y)\leq 3$ such that for every $(\gamma_{i,j})_{i=1,j=1}^{2,2}\in(0,1)^{4}\setminus\mathbf{E}_x$ and $(\lambda_{i,j})_{i=1,j=1}^{2,2}\in(0,1)^{4}\setminus\mathbf{E}_y$ with \eqref{eq:cond4corner} the following holds: for every self-affine measure $\mu$ corresponding to the IFS $\{F_1,F_2,F_3,F_4\}$ and the probability vector $\mathbf{p}=(p_{1},p_2,p_3,p_4)$
		$$\dim_{H}(\mu)=\begin{cases}
		    \dfrac{h_{\mathbf{p}}+\Phi_x(\mathbf{p})}{\chi_x(\mathbf{p})}-\dfrac{\Phi_x(\mathbf{p})}{\chi_y(\mathbf{p})} & \text{if }\chi_y(\mathbf{p})\geq\chi_x(\mathbf{p})\geq h_{\mathbf{p}}+\Phi_x(\mathbf{p})\\[10pt]
           1+\dfrac{h_{\mathbf{p}}-\chi_x(\mathbf{p})}{\chi_y(\mathbf{p})} & \text{if }\chi_y(\mathbf{p}),h_{\mathbf{p}}+\Phi_x(\mathbf{p})\geq\chi_x(\mathbf{p}), \\[10pt]
            \dfrac{h_{\mathbf{p}}+\Phi_y(\mathbf{p})}{\chi_y(\mathbf{p})}-\dfrac{\Phi_y(\mathbf{p})}{\chi_x(\mathbf{p})} & \text{if }\chi_x(\mathbf{p})\geq\chi_y(\mathbf{p})\geq h_{\mathbf{p}}+\Phi_y(\mathbf{p}),\\[10pt]
            1+\dfrac{h_{\mathbf{p}}-\chi_y(\mathbf{p})}{\chi_x(\mathbf{p})} & \text{if }\chi_x(\mathbf{p}),h_{\mathbf{p}}+\Phi_y(\mathbf{p})\geq\chi_y(\mathbf{p}),
		\end{cases}$$
		where \begin{align*}
		    \Phi_x(\mathbf{p})&=\sum_{k=0}^{\infty}\sum\limits_{q=0}^{k}\binom{k}{q}(p_3+p_4)\left(p_{1}^{q+1}p_{2}^{k-q}+p_2^{q+1}p_1^{k-q}\right)\log(\frac{q+1}{k+1})\\
            &\quad\quad+\sum_{k=0}^{\infty}\sum\limits_{q=0}^{k}\binom{k}{q}(p_1+p_2)\left(p_{3}^{q+1}p_{4}^{k-q}+p_4^{q+1}p_3^{k-q}\right)\log(\frac{q+1}{k+1}),
            \end{align*}
and
\begin{align*}
		    \Phi_y(\mathbf{p})&=\sum_{k=0}^{\infty}\sum\limits_{q=0}^{k}\binom{k}{q}(p_2+p_4)\left(p_{1}^{q+1}p_{3}^{k-q}+p_3^{q+1}p_1^{k-q}\right)\log(\frac{q+1}{k+1})\\
            &\quad\quad+\sum_{k=0}^{\infty}\sum\limits_{q=0}^{k}\binom{k}{q}(p_1+p_3)\left(p_{2}^{q+1}p_{4}^{k-q}+p_4^{q+1}p_2^{k-q}\right)\log(\frac{q+1}{k+1}).
\end{align*}
	\end{theorem}

\begin{proof}
    This is a direct consequence of Theorem~\ref{thm:main} and Feng and Hu's theorem \cite[Theorem~2.11]{FengandHu}.
\end{proof}

\begin{example}\label{ex:4corner} Theorem~\ref{thm:4corner} allows us to calculate the Hausdorff dimension of the generalised 4-corner set for Lebesgue typical points on a certain region of parameters. Let us further suppose that 
\begin{equation}\label{eq:cond4corner+}
\lambda_{1,1}\leq \gamma_{1,1},\ \lambda_{2,2}\leq\gamma_{2,2},\ \lambda_{1,2}\leq\gamma_{2,1}\text{ and }\lambda_{2,1}\leq\gamma_{1,2}.
\end{equation}
This implies that $\chi_x(\mathbf{p})\leq\chi_y(\mathbf{p})$ for any probability vector $\mathbf{p}$. Consider the probability vector
$$
p_1=\gamma_{1,1}\lambda_{1,1}^{s-1},\ p_2=\gamma_{1,2}\lambda_{2,1}^{s-1},\ p_3=\gamma_{2,1}\lambda_{1,2}^{s-1}\text{ and }p_4=\gamma_{2,2}\lambda_{2,2}^{s-1},
$$
where $s$ is the unique solution of the equation $\sum_{i=1}^2\gamma_{i,i}\lambda_{i,i}^{s-1}+\gamma_{i,3-i}\lambda_{3-i,i}^{s-1}=1$. It is well known that $\dim_HC\leq s$, see for example Falconer and Miao \cite{FalMia2007}. Thus, it is enough to show that the corresponding self-affine measure $\mu$ has Hausdorff dimension $s=1+\frac{h_{\mathbf{p}}-\chi_x(\mathbf{p})}{\chi_y(\mathbf{p})}$. For that, it is enough to verify that $h_{\mathbf{p}}+\Phi_x(\mathbf{p})\geq\chi_x(\mathbf{p})$ by Theorem~\ref{thm:4corner}. Lemma~\ref{lem:lowerbound} implies that this holds if
\begin{multline}\label{eq:suff}
\gamma_{1,1}\lambda_{1,1}^{s-1}\log(\frac{1-\gamma_{1,2}\lambda_{2,1}^{s-1}}{\lambda_{1,1}^{s-1}})+\gamma_{1,2}\lambda_{2,1}^{s-1}\log(\frac{1-\gamma_{1,1}\lambda_{1,1}^{s-1}}{\lambda_{2,1}^{s-1}})\\
+\gamma_{2,1}\lambda_{1,2}^{s-1}\log(\frac{1-\gamma_{2,2}\lambda_{2,2}^{s-1}}{\lambda_{1,2}^{s-1}})+\gamma_{2,2}\lambda_{2,2}^{s-1}\log(\frac{1-\gamma_{2,1}\lambda_{1,2}^{s-1}}{\lambda_{2,2}^{s-1}})>0.
\end{multline}
One can check numerically that a small neighborhood of the parameters $\gamma_{1,1}=\gamma_{2,2}=0.8$, $\gamma_{1,2}=\gamma_{2,1}=0.1$, $\lambda_{1,1}=\lambda_{2,2}=0.45$ and $\lambda_{2,1}=\lambda_{1,2}=0.09$ satisfy all three assumptions \eqref{eq:cond4corner}, \eqref{eq:cond4corner+} and \eqref{eq:suff}. For a visual representation of the generalised 4-corner set, see Figure~\ref{fig1}.
\end{example}

\begin{figure}
    \centering
    \includegraphics[width=0.48\linewidth]{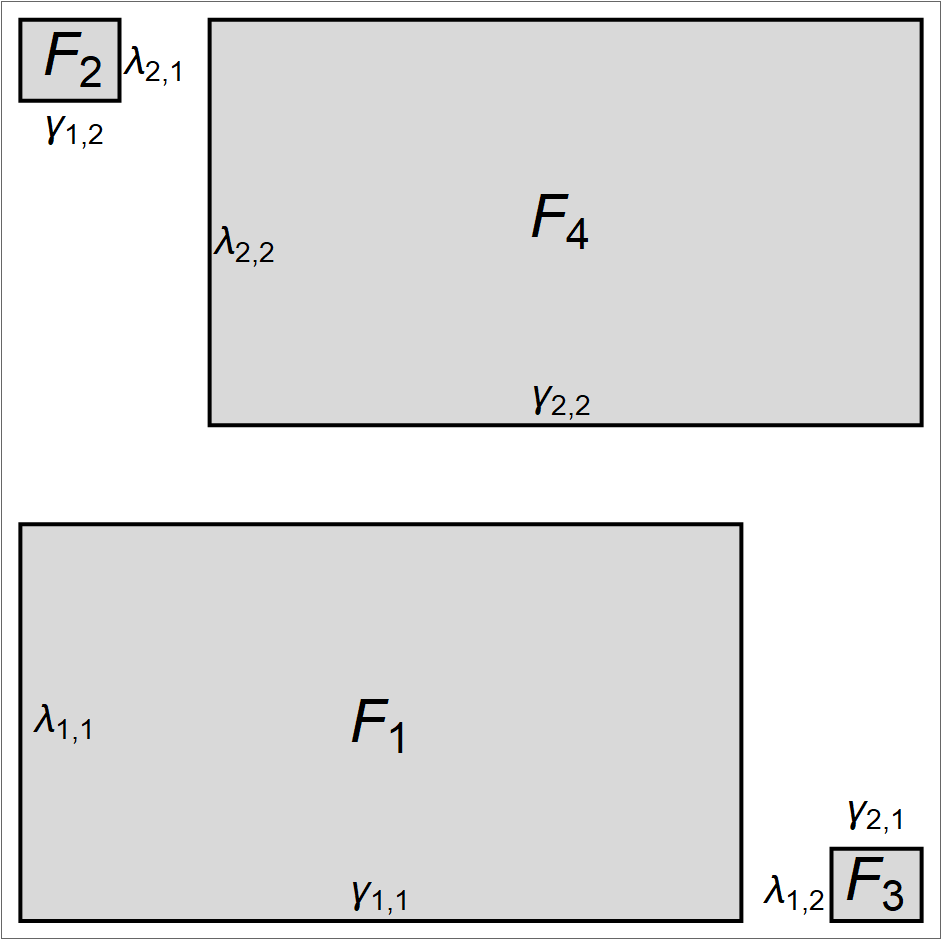}\hfill\includegraphics[width=0.48\linewidth]{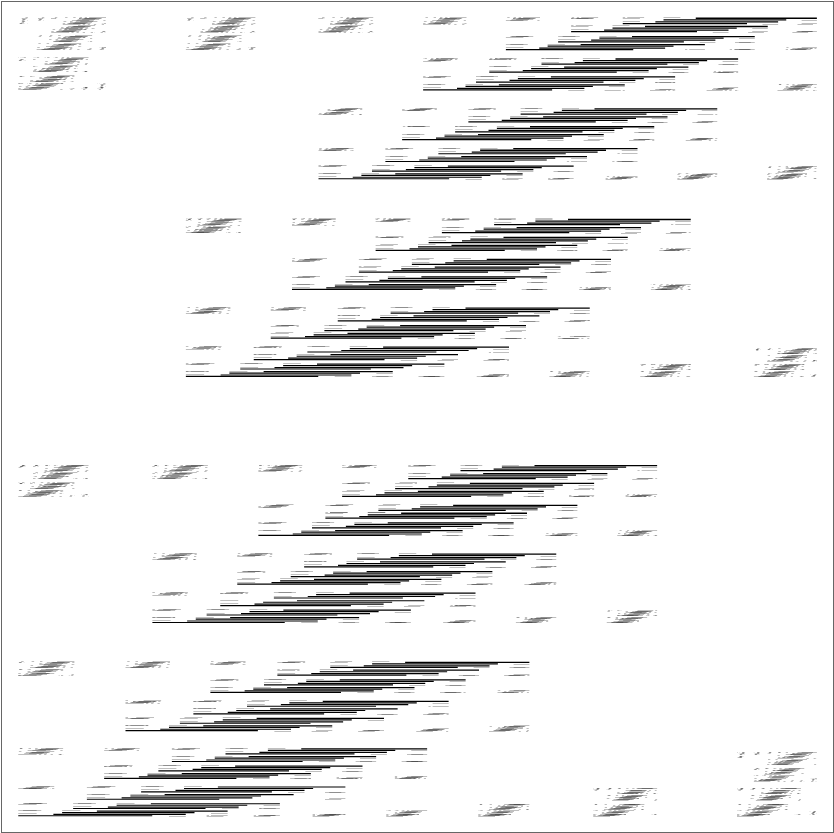}
    \caption{The first level cylinders (left) and the attractor (right) of the generalised 4-corner set in Example~\ref{ex:4corner}.}
    \label{fig1}
\end{figure}

\subsection{Outline of the argument:} {Here, we provide an outline of the arguments of our main results, Theorem \ref{thm:main} and Theorem \ref{dimattractor}. The self-similar IFS $\mathcal{S}_{\underline{\lambda},\underline{t}}$ we consider has many exact overlaps at each level because many maps share the same fixed point. This implies our IFS does not satisfy the exponential separation condition. Thus, we cannot apply Theorem~\ref{MainHochmanresult} by Hochman \cite{Hochman2014} directly. However, the weaker version, the weak exponential separation condition (see Definition \ref{WESC}) allows exact overlaps, and by a direct consequence of Hochman's seminal work \cite{Hochman2014}, Breuillard and Varj\'u \cite{BRVAR} determined the Hausdorff dimension of the self-similar measures in terms of random walk entropy under the WESC. However, they provided the proof only for the case of homogeneous similarity ratios. In the next section (Section \ref{Sec3Pre}), following the arguments of Breuillard and Varj\'u \cite{BRVAR}, we provide a proof for Hausdorff dimension of the general self-similar measure under the WESC (see Theorem \ref{MRandom}). This result is one of the main ingredients for proving our main result Theorem \ref{thm:main}. In the section \ref{ProofWESCforCFS}, we define a stronger version of WESC, namely the exponential separation condition for common fixed point systems (ESC for CFS, see Definition \ref{DefWESCCFS}), which is compatible with our self-similar IFS $\mathcal{S}_{\underline{\lambda},\underline{t}}$. In that section, we prove the self-similar IFS $\mathcal{S}_{\underline{\lambda},\underline{t}}$ satisfies ESC for CFS for typical parameters. One key fact is that ESC for CFS provides a complete description of the exactly overlapping cylinders in the symbolic space. This is a significant advantage for computing the random walk entropy, which is otherwise a highly challenging task. In the section \ref{sec:rw}, we compute the value of the random walk entropy for the self-similar IFS $\mathcal{S}_{\underline{\lambda},\underline{t}}$ under ESC for CFS. Now, by combining the results of Section \ref{ProofWESCforCFS} and \ref{sec:rw}, we get Theorem \ref{thm:main}. In Section \ref{resultforattrac}, we prove another of our main theorems, Theorem \ref{dimattractor}. To prove that, we define the associated graph-directed self-similar IFS corresponding to our IFS $\mathcal{S}_{\underline{\lambda},\underline{t}}$. We show that if IFS $\mathcal{S}_{\underline{\lambda},\underline{t}}$ satisfies ESC for CFS, then the associated graph-directed self-similar IFS satisfies the graph-directed version of the ESC. We also prove that, under the ESC, the Hausdorff dimension of the attractor of the associated graph-directed self-similar IFS is determined by the unique exponent for which the spectral radius of the corresponding connection matrix is equal to $1$. This result provides the required lower bound for Theorem \ref{dimattractor}, which coincides with the general upper bound. 

}

	\section{Preliminaries}\label{Sec3Pre}

Before we turn to the proof of our main results, we need to introduce some notations and make some preliminary statements.

\subsection{Symbolic space} 
Let $\Sigma:=\{1,2,\dots,N\}^{\mathbb{N}}$ be the set of all infinite sequences with symbols from $\{1,2,\dots,N\}$. The set $\Sigma$ is the symbolic space corresponding to the IFS $\mathcal{I}=\{f_1,\ldots,f_N\}$. Let $\mathbf{i}=i_{1}i_{2}\dots \in \Sigma$. We define $\mathbf{i}_{|_m}:=i_{1}i_{2}\dots i_{m}$ for all $m\in \mathbb{N}$. We denote the set of all finite sequences of length $n$ with symbols from $\{1,2,\dots, N\}$ by $\Sigma_{n}$. Set $\Sigma^{*}:=\bigcup_{n=1}^\infty\Sigma_{n}.$ The notation $|\mathbf{i}|$ denotes the length of the finite sequence $\mathbf{i}\in \Sigma^*$. For a symbol $i\in\{1,\dots,N\}$ and a finite word $\mathbf{i}\in\Sigma^*$, let $\#_i\mathbf{i}$ be the number of the appearances of the symbol $i$ in the word $\mathbf{i}$.

If $\mathbf{i}=i_1 i_2 \dots i_n\in \Sigma^{*}$ and $\mathbf{j}=j_1 j_2\dots\in \Sigma^{*} \cup \Sigma$, then let $\mathbf{i}*\mathbf{j}=i_1 i_2 \dots i_n j_1 j_2\dots$ be the concatenation of $\mathbf{i}$ and $\mathbf{j}.$ The symbolic space $\Sigma$ equipped with metric $\rho$ is a compact metric space, where the metric $\rho$ is defined as follows
$$\rho(\mathbf{i},\mathbf{j})=2^{-|\mathbf{i}\wedge\mathbf{j}|}$$
for $\mathbf{i},\mathbf{j}\in \Sigma$, where $\mathbf{i}\wedge\mathbf{j}$ denotes the initial largest common segment of $\mathbf{i}$ and $\mathbf{j}$. For $\mathbf{i}=i_1 i_2 \dots i_n\in \Sigma^{*},$ let $f_{\mathbf{i}}=f_{i_1}\circ f_{i_2}\circ \dots \circ f_{i_n}$ and $r_{\mathbf{i}}=r_{i_1} r_{i_2}\dots r_{i_n}.$ The mapping $\Pi: \Sigma \to A$, defined by $$\Pi(\mathbf{i})=\lim_{n\to\infty}f_{i_1}\circ f_{i_2}\circ \dots \circ f_{i_n}(0)=\lim_{n\to \infty}f_{\mathbf{i}_{|_n}}(0),$$
is called the natural projection corresponding to the IFS $\mathcal{I}.$ For $\mathbf{i}=i_1 i_2 \dots i_n\in \Sigma^{*},$ set $\Pi(\mathbf{i}):=f_{i_1}\circ f_{i_2}\circ \dots \circ f_{i_n}(0).$ \par

\subsection{Random walk entropy}

In this section, we provide a more detailed description of the method by Hochman \cite{Hochman2014}.

\begin{definition}
    Let $(X,\mathcal{B},\nu)$ be a probability space and $\xi=\{C_1,C_2,\dots,C_m\}$ be a measurable partition, i.e.  $\cup_{i=1}^mC_{i}=X,$ \(C_i\in \mathcal{B}\) and $C_{i}\cap C_{j}=\emptyset$ for all $i\ne j\in \{1,2,\dots,m\}$. Then, the \texttt{Shannon entropy} $H$ of the measure $\nu$ with respect to the partition $\xi$ is defined as follows
    $$H(\nu,\xi)=-\sum_{i=1}^{m}\nu(C_{i})\log(\nu(C_{i})).$$
\end{definition}

\par
The self-similar IFS $\mathcal{I}$ with the probability vector $\mathbf{p}$ is known as a probabilistic self-similar IFS. We denote the probabilistic self-similar IFS by $(\mathcal{I},\mathbf{p}).$ 
\begin{definition}
    The \texttt{random walk entropy} of the self-similar measure $\mu$ corresponding to the probabilistic self-similar IFS $(\mathcal{I},\mathbf{p})$ is defined as 
$$h_{RW}(\mu):=\lim_{n\to\infty}\frac{H_{n}}{n},$$
    where $H_{n}$ is the Shannon entropy of the measure $\sum_{\mathbf{i}\in \Sigma_{n}}p_{\mathbf{i}}{\delta}_{f_{\mathbf{i}}},$ which is as follows
    $$H_{n}=-\sum_{\mathbf{i}\in \Sigma_{n}}p_{\mathbf{i}}\log\bigg(\sum_{\substack{\mathbf{j}\in \Sigma_{n}\\ f_{\mathbf{i}}\equiv f_{\mathbf{j}}}}p_\mathbf{j}\bigg).$$
\end{definition}


The random walk entropy provides a more sophisticated upper bound on the dimension of the self-similar measures if exact overlaps occur. In particular, for every self-similar measure $\mu$ corresponding to the probabilistic self-similar IFS $(\mathcal{I},\mathbf{p})$, 
$$\dim_{H}(\mu)\leq \min\bigg\{1,\frac{h_{RW}(\mu)}{\chi(\mathbf{p})}\bigg\},$$
see \cite[Theorem~3.2.7]{BSS23}.

We now consider some of the following notation from Hochman's work \cite{Hochman2014}.
 For $n\in \mathbb{N},$ $\mathcal{D}_{n}$ denote the partition of $\mathbb{R}$ into intervals of length $2^{-n},$ i.e. $$\mathcal{D}_{n}=\bigg\{\bigg[\frac{k}{2^n},\frac{k+1}{2^n}\bigg): k\in \mathbb{Z}\bigg\}.$$ The notation $H(\nu, \mathcal{D}_n)$ denotes the Shannon entropy of the probability measure $\nu$ on $\mathbb{R}$ with respect to the partition $\mathcal{D}_n$. For $m<n$, $H(\nu, \mathcal{D}_n|\mathcal{D}_m)$ denotes the conditional entropy and is defined as $H(\nu, \mathcal{D}_n|\mathcal{D}_m):= H(\nu, \mathcal{D}_n)-H(\nu,\mathcal{D}_m).$ \par 
 
 The  probability measure $\tilde{\nu}^{(n)}$ on $\mathbb{R}\times\mathbb{R}$ is defined by
$$\tilde{\nu}^{(n)}=\sum_{\mathbf{i}\in I^n}p_{\mathbf{i}}\delta_{(f_{\mathbf{i}}(0),r_{\mathbf{i}})},$$
where $\delta_x$ denotes the Dirac measure on the point $x$. We denote  $\tilde{\mathcal{D}}_n=\mathcal{D}_n\times \mathcal{F}$ is a partition on $\mathbb{R}\times \mathbb{R}$, where $\mathcal{F}$ is the partition of $\mathbb{R}$ into points. 
\begin{theorem}\cite{Hochman2014}\label{Hochmanentropy}
  Let $\mu$ be the self-similar measure, and the probability measure $\tilde{\nu}^{(n)}$ be defined as above. If $\dim_{H}(\mu)<1$, then
  $$\lim_{n\to \infty}\frac{1}{n'}H(\tilde{\nu}^{(n)},\tilde{\mathcal{D}}_{qn'}|\tilde{\mathcal{D}}_{n'})=0\quad  \forall \quad  q>1,$$
  where $n'=\lfloor n \log(\frac{1}{r})\rfloor$ and $r= \prod_{i=1}^{N}{r_{i}}^{p_i}.$ 
\end{theorem}
\begin{definition}\label{WESC}
   We say that IFS $\mathcal{I}$ satisfies the \texttt{weak exponential separation condition (WESC)}, if there exists a sequence $\{n_{k}\}$ and 
$\exists~ b>1 $ such that $n_{k}\to \infty$ and $\forall~ \mathbf{i},\mathbf{j}\in \Sigma_{n_k}$, we have the following:
$$\text{either}~r_{\mathbf{i}}\ne r_{\mathbf{j}}$$ $$\text{or}~ r_{\mathbf{i}}=r_{\mathbf{j}}~\&~f_{\mathbf{i}}(0)=f_{\mathbf{j}}(0)\iff f_{\mathbf{i}}\equiv f_{\mathbf{j}}$$
$$\text{or}~ r_{\mathbf{i}}=r_{\mathbf{j}}~\&~|f_{\mathbf{i}}(0)-f_{\mathbf{j}}(0)|>2^{-n_{k}' b}.$$ 
\end{definition}
The following result is a direct consequence of Theorem~\ref{Hochmanentropy}, although it was not stated explicitly in Hochman's seminal paper \cite{Hochman2014}. It was observed by Breuillard and Varj\'u \cite{BRVAR}, and they provided a proof for the case of homogeneous contraction ratios. We provide here the proof on the general case for completeness following Breuillard and Varj\'u~\cite{BRVAR}.
 

\begin{theorem}\label{MRandom}
 Let $\mu$ be a self-similar measure corresponding to the probabilistic self-similar IFS $(\mathcal{I},\mathbf{p})$. If IFS $\mathcal{I}$ satisfy WESC, then 
$$\dim_{H}(\mu)=\min\bigg\{1,\frac{h_{RW}(\mu)}{\chi(\mathbf{p})}\bigg\},$$
where $h_{RW}(\mu)$ is the random walk entropy for the self-similar measure $\mu$. 
\end{theorem}

\begin{proof}
Since $\dim_{H}(\mu)\leq \frac{h_{RW}(\mu)}{\chi(\mathbf{p})}$, we may assume that $\dim_{H}(\mu)<1.$ By Theorem \ref{Hochmanentropy}, we have
 $$\lim_{n\to \infty}\frac{1}{n'}H(\tilde{\nu}^{(n)},\tilde{\mathcal{D}}_{qn'}|\tilde{\mathcal{D}}_{n'})=0\quad  \forall \quad  q>1,$$
  where $n'=\lfloor n \log(\frac{1}{r})\rfloor$  and $r= \prod_{i=1}^{N}{r_{i}}^{p_i}$. Similarly to Hochman \cite{Hochman2014}, since $\mu$ is exact dimensional, we get that 
  $$\lim_{n\to \infty}\frac{1}{n'}H(\tilde{\nu}^{(n)},\tilde{\mathcal{D}}_{n'})=\dim_{H}(\mu).$$ Thus, for proving our result, it remains only to show that
  \begin{equation}\label{eq3}
    \lim_{n\to \infty}\frac{1}{n'}H(\tilde{\nu}^{(n)},\tilde{\mathcal{D}}_{qn'})\geq \frac{h_{RW}(\mu)}{\chi(\mathbf{p})}  
  \end{equation}
for some $q>1.$
 Now, we will prove the equality \eqref{eq3}. Let $q=b$ from Definition~\ref{WESC}. For each $a\in \mathbb{R}$, we have
 $$\tilde{\nu}^{(n)}\bigg(\bigg[\frac{k}{2^{bn'}},\frac{k+1}{2^{bn'}}\bigg)\times \{a\}\bigg)=\sum_{\substack{{\mathbf{j}}\in \Sigma_{n}\\ r_{\mathbf{j}}=a}}p_{\mathbf{j}}\delta_{f_{\mathbf{j}}(0)}\bigg[\frac{k}{2^{bn'}},\frac{k+1}{2^{bn'}}\bigg).$$ Thus, the Shannon entropy of the probability measure $\tilde{\nu}^{(n_k)}$ on $\mathbb{R}\times \mathbb{R} $ with respect to the partition $\tilde{\mathcal{D}}_{bn_{k}'}$ is
 \begin{align*}
H(\tilde{\nu}^{(n_k)},\tilde{\mathcal{D}}_{bn_{k}'})&=-\sum_{\substack{a\in \mathbb{R}\\ \exists~ r_{\mathbf{i}}=a}}\sum_{l\in \mathbb{Z}} \tilde{\nu}^{(n_k)}\bigg(\bigg[\frac{l}{2^{bn_{k}'}},\frac{l+1}{2^{bn_{k}'}}\bigg)\times \{a\}\bigg) \log\bigg(\tilde{\nu}^{(n_k)}\bigg(\bigg[\frac{l}{2^{bn_{k}'}},\frac{l+1}{2^{bn_{k}'}}\bigg)\times \{a\}\bigg)\bigg)\\
&=-\sum_{\substack{a\in \mathbb{R}\\ \exists~ r_{\mathbf{i}}=a}}\sum_{l\in \mathbb{Z}}\bigg(\sum_{\substack{{\mathbf{j}}\in \Sigma_{n_{k}}\\ r_{\mathbf{j}}=a}}p_{\mathbf{j}}\delta_{f_{\mathbf{j}}(0)}\bigg[\frac{l}{2^{bn_{k}'}},\frac{l+1}{2^{bn_{k}'}}\bigg)\bigg) \log\bigg(\sum_{\substack{{\mathbf{j}}\in \Sigma_{n_{k}}\\ r_{\mathbf{j}}= a}}p_{\mathbf{j}}\delta_{f_{\mathbf{j}}(0)}\bigg[\frac{l}{2^{bn_{k}'}},\frac{l+1}{2^{bn_{k}'}}\bigg)\bigg).
\end{align*}
Since the IFS $\mathcal{I}$ satisfies the WESC, we have 
\begin{align*}
H(\tilde{\nu}^{(n_k)},\tilde{\mathcal{D}}_{bn_{k}'})&=-\sum_{\substack{a\in \mathbb{R}\\ \exists~ r_{\mathbf{i}}=a}}\bigg(\sum_{\substack{{\mathbf{j}}\in \Sigma_{n_k}\\ f_{\mathbf{i}}\equiv f_{\mathbf{j}}}}p_{\mathbf{j}}\bigg)\log\bigg(\sum_{\substack{{\mathbf{j}}\in \Sigma_{n_{k}}\\ f_{\mathbf{i}}\equiv f_{\mathbf{j}}}}p_{\mathbf{j}}\bigg)\\&= -\sum_{{\mathbf{i}}\in I^{n_{k}}}p_{\mathbf{i}} \log\bigg(\sum_{\substack{{\mathbf{j}}\in \Sigma_{n_k}\\ f_{\mathbf{i}}\equiv f_{\mathbf{j}}}}p_{\mathbf{j}}\bigg).  
\end{align*}
This implies that
 \begin{equation*}
    \lim_{k\to \infty}\frac{1}{n_{k}'}H(\tilde{\nu}^{(n_k)},\tilde{\mathcal{D}}_{bn_{k}'})\geq \frac{h_{RW}(\mu)}{\chi(\mathbf{p})},  
  \end{equation*}
  which completes the proof.
\end{proof}

\section{Exponential separation condition for common fixed point system}\label{ProofWESCforCFS}

{First, we provide a brief explanation and outline of this section. Here, our aim is to show that the self-similar IFS $\mathcal{S}_{\underline{\lambda},\underline{t}}$ satisfies a particular version of WESC for typical parameters, namely, the Exponential Separation Condition for the Common Fixed Point System (ESC for CFS), which we introduce in subsection \ref{sec:block}. Our strategy is to first establish it for two-fixed-point systems and then extend it to general fixed-point systems. In the first subsection \ref{Sec4.1Tech}, we provide two lemmas for general self-similar IFS $\mathcal{F}$ \eqref{eq:FIFS} on $\mathbb{R}$. The IFS $\mathcal{F}$ comprises two classes of mappings, those whose fixed point is $0$ and those whose fixed points are nonzero. A crucial property is that the images of the maps with nonzero fixed points remain bounded away from $0$.
The first block of an infinite sequence $\mathbf{i}$ is the longest initial segment of $\mathbf{i}$ for which the maps corresponding to all symbols in the segment have the same fixed point. The first lemma, Lemma \ref{GenExponentialA1}, proves that the projection of any pair of the infinite sequences $\mathbf{i}$ and $\mathbf{j}$ are exponentially separated if the first blocks are pairwise disjoint, all the maps in the first blocks having same fixed point $0$ and the similarity ratios corresponding to the first blocks are not close enough.
Furthermore, if the similarity ratios corresponding to the first blocks are sufficiently close and the lengths of the first blocks are sufficiently large, then Lemma \ref{GenFirstorderTrans} shows that the first-order transversality condition holds. 

In the subsection \ref{sec:block}, we introduce the block structure for self-similar IFS with a common fixed point structure, and introduce the Exponential Separation Condition for the Common Fixed Point System (ESC for CFS). In the subsection \ref{2fixedpointsys}, we consider a self-similar IFS $\mathcal{J}$ (see \ref{eq:IFS2fix}) having only two fixed points $0$ and $1$. We define a set $\mathcal{B}$, consisting of all pairs of infinite sequences from the symbolic space whose first blocks are disjoint (see \eqref{firstsetB}). 
Furthermore, we partition the set $\mathcal{B}$ into two disjoint subsets: $\mathcal{B}_1$
 , consisting of pairs whose first blocks have different fixed points, and $\mathcal{B}_2$
 , consisting of pairs whose first blocks have the same fixed point (see \eqref{eqBdiv1}). 
Furthermore, we partition the set $\mathcal{B}_2$
  into two disjoint subsets: $\mathcal{B}_3$
 , consisting of pairs whose first blocks have small length, and $\mathcal{B}_4$
 , consisting of pairs whose first blocks have large length (see \eqref{eqBdiv2}). Since IFS $\mathcal{J}$ is a special form of $\mathcal{F}$ defined in the subsection \ref{Sec4.1Tech} and the sets $\tilde{\mathcal{A}}_{1}$ and $\tilde{\mathcal{A}}_{2}$ provide a disjoint partition of the set  $\mathcal{B}_4$. Thus, by applying Lemma \ref{GenExponentialA1} and Lemma \ref{GenFirstorderTrans}, we conclude that for every pair in $\mathcal{B}_4$, either the projections are exponentially separated or the first-order transversality condition holds (see Lemma \ref{ExponentialA1b}). The  $\mathcal{B}_3$ is a compact subset of $\Sigma\times\Sigma$, and for any pair in $\mathcal{B}_3$ the projection is non-degenerate. Thus, by applying the idea of Hochman \cite{Hochman2022}, we show that the higher-order transversality condition holds for the set $\mathcal{B}_3$ (see Lemma \ref{Trans3}). We apply a similar idea to the set $\mathcal{B}_2.$ By using these transversality conditions, we show that IFS $\mathcal{J}$ satisfies ESC for CFS outside of a Hausdorff co-dimension one exceptional set of natural parameters (see Proposition \ref{ESC for gerenalsystem}). In the subsection  \ref{Genaralcaseextend}, we will prove our aim for a general fixed point system by using results on two fixed point systems. The key idea is that, upon differentiating the general IFS $\mathcal{S}_{\underline{\lambda},\underline{t}}$
  with respect to the fixed point $t_i$, we obtain a self-similar IFS $\mathcal{S}_i$, which have only two fixed points $0$ and $1$. The natural projection corresponding to the IFS $\mathcal{S}_{\underline{\lambda},\underline{t}}$ and the natural projection corresponding to the IFS $\mathcal{S}_i$ are well-connected by a formula (see \ref{connectionformula}). Finally, by combining this formula with the result for the two-fixed-point system, we obtain the desired conclusion for the general case (see Proposition \ref{ESC for gerenalsystem}). 
}

\subsection{Technical lemmas}\label{Sec4.1Tech}

 We begin by discussing some general results for later purposes, which will be particularly useful in describing the dimension theory of self-similar IFS with a common fixed-point structure. 
 
 Let us consider a self-similar IFS $\mathcal{F}$ on $\mathbb{R}$, 
 \begin{equation}\label{eq:FIFS}
     \mathcal{F}:=\{f_{i}(x)=\lambda_{i}x\}_{i=1}^{M_1}\cup\{f_{i}(x)=\lambda_{i}x+c_i\}_{i=M_1+1}^{M_2},
 \end{equation}
 where $M_2>M_1\geq 2$. Suppose that there exist $\epsilon,A>0$ such that 
\begin{enumerate}
    \item\label{it:1} $\lambda_{i}\in [\epsilon,1-\epsilon]~\forall~i\in \{1,2,\dots,M_1\}$ and $|\lambda_{i}|\in [\epsilon,1-\epsilon]~\forall~i\in \{M_{1}+1,M_{1}+2,\dots,M_2\}$,
    \item\label{it:2} $f_{i}[0,A]\subseteq [\epsilon,A] ~\forall~i\in \{M_{1}+1,\dots,M_2\}$.
 \end{enumerate}
Let $\Sigma=\{1,\ldots,M_2\}^\N$ be the symbolic space and let $\Pi$ be the natural projection corresponding to the IFS $\mathcal{F}$. Throughout the section, we assume that $A>\epsilon>0$ are arbitrary but fixed.

For an $\mathbf{i}\in\Sigma$, we define the ``first block" $b_1^{\mathbf{i}}$ of $\mathbf{i}$ as follows: if $i_1\geq M_1+1$ then $b_1^{\mathbf{i}}=i_1$. Otherwise, let $|b_1^{\mathbf{i}}|:=\min\{k\geq1:i_k\geq M_1+1\}-1$, and if $|b_1^{\mathbf{i}}|=\infty$ then let $b_1^{\mathbf{i}}=\mathbf{i}$, otherwise let $b_1^{\mathbf{i}}:=\mathbf{i}_{|_{|b_1^{\mathbf{i}}|}}$.

We say that for $\mathbf{i},\mathbf{j}\in\Sigma$, the first blocks are disjoint if the sets formed by the symbols in the first blocks $b_1^{\mathbf{i}}$ and $b_1^{\mathbf{j}}$ are disjoint. We denote it by $b_{1}^{\mathbf{i}}\cap b_{1}^{\mathbf{j}}=\emptyset$. In other words, $\min\{\#_ib_{1}^{\mathbf{i}},\#_ib_{1}^{\mathbf{j}}\}=0$ for every $1\leq i\leq M_2$.

In the next two lemmas, we will follow and modify techniques of Kamalutdinov and Tetenov \cite{MR3843511}.

We define a set as follows $$\tilde{\mathcal{A}}_{1}=\bigg\{(\mathbf{i},\mathbf{j})\in \Sigma\times \Sigma: 1\leq i_1,j_1\leq M_1~\&~b_{1}^{\mathbf{i}}\cap b_{1}^{\mathbf{j}}=\emptyset~\&~ \frac{\lambda_{b_{1}^{\mathbf{i}}}}{\lambda_{b_{1}^{\mathbf{j}}}}\notin \bigg(\frac{\epsilon}{2A},\frac{2A}{\epsilon }\bigg)~\&~b_{1}^{\mathbf{i}}\ne \mathbf{i},b_{1}^{\mathbf{j}}\ne \mathbf{j}\bigg\}.$$ 
\begin{lemma}\label{GenExponentialA1}
    Let $\mathcal{F}$ be the self-similar IFS of the form \eqref{eq:FIFS} with natural projection $\Pi$. Let $A>\epsilon>0$. Then for every $\underline{\lambda}=(\lambda_i)_{i=1}^{M_2}$ satisfying \eqref{it:1}--\eqref{it:2} and for every pair $(\mathbf{i},\mathbf{j})\in \tilde{\mathcal{A}}_{1}$,
    $$|\Pi(\mathbf{i})-\Pi(\mathbf{j})|\geq {\epsilon}^{3\max\{|b_{1}^{\mathbf{i}}|,|b_{1}^{\mathbf{j}}|\}}.$$
\end{lemma}
\begin{proof} Without loss of generality, we assume that $\frac{\lambda_{b_{1}^{\mathbf{i}}}}{\lambda_{b_{1}^{\mathbf{j}}}}\geq\frac{2 A}{\epsilon},$ the other case can be proven similarly. Since $b_{1}^{\mathbf{i}}\ne \mathbf{i},b_{1}^{\mathbf{j}}\ne \mathbf{j}$ and $f_{i}[0,A]\subseteq [\epsilon,A] ~\forall~i\in \{M_{1}+1,\dots,M_2\}$, we have 
  $$\Pi(\mathbf{i})\in f_{b_{1}^{\mathbf{i}}}([\epsilon,A])~~\text{and}~~\Pi(\mathbf{j})\in f_{b_{1}^{\mathbf{j}}}([\epsilon,A]).$$
  Thus, we get
  \begin{align*}
    \Pi(\mathbf{i})-\Pi(\mathbf{j})=\lambda_{b_{1}^{\mathbf{i}}}x- \lambda_{b_{1}^{\mathbf{j}}}y
  \end{align*}
  for some $x,y\in [\epsilon,A]$. This implies that 
  \begin{align*}
   \Pi(\mathbf{i})-\Pi(\mathbf{j})&\geq \lambda_{b_{1}^{\mathbf{i}}} \epsilon- \lambda_{b_{1}^{\mathbf{j}}}A=\lambda_{b_{1}^{\mathbf{i}}} \epsilon
   \bigg(1-\frac{\lambda_{b_{1}^{\mathbf{j}}} A}{\lambda_{b_{1}^{\mathbf{i}}} \epsilon}\bigg)\geq \frac{\epsilon}{2} \lambda_{b_{1}^{\mathbf{i}}}\geq {\epsilon}^{3|b_{1}^{\mathbf{i}}|}.
  \end{align*}
\end{proof}

Let $N_{0}=\lceil \frac{(1-\epsilon)(2A+\epsilon)}{\epsilon^3} \rceil+1.$ We define a set as follows 
\begin{align*}
  \tilde{\mathcal{A}}_{2}=\bigg\{(\mathbf{i},\mathbf{j})\in \Sigma\times \Sigma: 1\leq i_1,j_1\leq M_1~\&~b_{1}^{\mathbf{i}}\cap b_{1}^{\mathbf{j}}=\emptyset~\&~ \frac{\lambda_{b_{1}^{\mathbf{i}}}}{\lambda_{b_{1}^{\mathbf{j}}}}\in \bigg(\frac{\epsilon}{2A},\frac{2A}{\epsilon}\bigg)~\&~b_{1}^{\mathbf{i}}\ne \mathbf{i},b_{1}^{\mathbf{j}}\ne \mathbf{j}\\~\&~\max_{1\leq k\leq M_1}\bigg\{\max\{\#_kb_{1}^{\mathbf{i}},\#_kb_{1}^{\mathbf{j}}\}\bigg\}>N_{0}\bigg\}  
\end{align*}

\begin{lemma}\label{GenFirstorderTrans}
Let $\mathcal{F}$ be the self-similar IFS of the form \eqref{eq:FIFS} with natural projection $\Pi$. Let $A>\epsilon>0$. Then there exists a constant $C>0$ such that for every $\underline{\lambda}=(\lambda_i)_{i=1}^{M_2}$ satisfying \eqref{it:1}--\eqref{it:2} and for every $(\mathbf{i},\mathbf{j})\in \tilde{\mathcal{A}}_{2}$ 
   $$\bigg{|}\frac{\partial(\Pi(\mathbf{i})-\Pi(\mathbf{j}))}{\partial \lambda_{k}}\bigg{|}\geq C{\epsilon}^{\max\{|b_{1}^{\mathbf{i}}|,|b_{1}^{\mathbf{j}}|\}},$$ where $k$ is such that $\max\{\#_kb_{1}^{\mathbf{i}},\#_kb_{1}^{\mathbf{j}}\}>N_0$.
\end{lemma}
\begin{proof} Let $(\mathbf{i},\mathbf{j})\in \tilde{\mathcal{A}}_{2}$. For simplicity, let $m_k=\#_kb_{1}^{\mathbf{i}}$ and $n_k=\#_kb_{1}^{\mathbf{j}}$ for $k\in \{1,2,\dots,M_1\}$. Without loss of generality, we assume that $m_r>N_0$ for some $r\in \{1,2\dots, M_1\}$. By definition, $n_r=0$. Let $\lambda_{r}'\in [\epsilon, 1-\epsilon]$ be such that $\lambda_r<\lambda_r'$ and
$$\frac{\epsilon}{2A} \leq \frac{(\lambda_r')^{m_r}\prod_{i=1,\ i\neq r}^{M_1}\lambda_{i}^{m_i}}{\prod_{i=1}^{M_1}\lambda_{i}^{n_i}}\leq \frac{2A}{\epsilon}.$$ 
Let $\delta=\lambda_{r}'-\lambda_{r}>0$. Then, by using the mean value theorem, we have
$$m_r \lambda_{r}^{m_{r}-1}\leq\frac{(\lambda_{r}')^{m_r}-\lambda_{r}^{m_r}}{\lambda_{r}'-\lambda_{r}}=\frac{|(\lambda_{r}')^{m_r}-\lambda_{r}^{m_r}|}{\delta}\leq m_r (\lambda_{r}')^{m_{r}-1}.$$
Now, we define another self-similar IFS $$\mathcal{F}'=\{\mathcal{F}\setminus\{f_{r}(x)=\lambda_{r} x\}\}\cup\{f'_{r}(x)=\lambda_{r}' x\}.$$
The IFS $\mathcal{F}'$ is just obtained from the IFS $\mathcal{F}$ by replacing the map $f_{r}$ by $f_{r}'$. Let $\Pi'$ be the natural projection corresponding to the IFS $\mathcal{F}'$. Thus, by the displacement lemma by Kamalutdinov and Tetenov \cite[Theorem~17]{MR3843511}, we have   
\begin{equation}\label{eq:disp}
|\Pi(\mathbf{i})-\Pi'(\mathbf{i})|\leq \frac{\delta}{\epsilon}\quad \forall\quad  \mathbf{i}\in \Sigma.
\end{equation}
For the given pair $(\mathbf{i},\mathbf{j})\in \Sigma\times \Sigma$, we have
$$\Pi(\mathbf{i})=f_{1}^{m_1}\circ \dots \circ f_{M_1}^{m_{M_1}}(f_{\tilde{i}}(\Pi(\hat{\mathbf{i}}))),~~ \Pi(\mathbf{j})=f_{1}^{n_1}\circ \dots\circ f_{M_1}^{n_{M_1}}(f_{\tilde{j}}(\Pi(\hat{\mathbf{j}})))$$
$$\Pi'(\mathbf{i})=f_{1}^{m_1}\circ \dots \circ f_{r-1}^{m_{r-1}}\circ {f}_{{r}}'^{m_{r}}\circ f_{{r+1}}^{m_{r+1}}\circ \dots \circ f_{M_1}^{m_{M_1}}(f_{\tilde{i}}(\Pi'(\hat{\mathbf{i}}))), $$
$$\Pi'(\mathbf{j})=f_{1}^{n_1}\circ\dots\circ f_{M_1}^{n_l}(f_{\tilde{j}}(\Pi'(\hat{\mathbf{j}}))),$$
for some $\tilde{i},\tilde{j}\in \{M_{1}+1,\dots,M_2\}$ and $\hat{\mathbf{i}},\hat{\mathbf{j}}\in \Sigma.$ Now, we will estimate the expression $|(\Pi(\mathbf{i})-\Pi(\mathbf{j}))-(\Pi'(\mathbf{i})-\Pi'(\mathbf{j}))|$. We have 
\begin{align*}
  &|(\Pi(\mathbf{i})-\Pi(\mathbf{j}))-(\Pi'(\mathbf{i})-\Pi'(\mathbf{j}))|\\
  &\quad=|\Pi(\mathbf{i})-\Pi'(\mathbf{i})+\Pi'(\mathbf{j})-\Pi(\mathbf{j})|\\
  &\quad\geq \left|f_{1}^{m_1}\circ \dots \circ f_{M_1}^{m_{M_1}}(f_{\tilde{i}}(\Pi(\hat{\mathbf{i}})))-f_{1}^{m_1}\circ \dots \circ f_{r-1}^{m_{r-1}}\circ {f}_{{r}}'^{m_{r}}\circ f_{{r+1}}^{m_{r+1}}\circ \dots \circ f_{M_1}^{m_{M_1}}(f_{\tilde{i}}(\Pi'(\hat{\mathbf{i}})))\right|\\
  &\quad\quad -|\Pi'(\mathbf{j})-\Pi(\mathbf{j})|\\
  &\quad\geq\left|f_{1}^{m_1}\circ \dots \circ f_{M_1}^{m_{M_1}}(f_{\tilde{i}}(\Pi'(\hat{\mathbf{i}})))-f_{1}^{m_1}\circ \dots \circ f_{r-1}^{m_{r-1}}\circ {f}_{{r}}'^{m_{r}}\circ f_{{r+1}}^{m_{r+1}}\circ \dots \circ f_{M_1}^{m_{M_1}}(f_{\tilde{i}}(\Pi'(\hat{\mathbf{i}})))\right|\\
  &\quad\quad-\left|f_{1}^{m_1}\circ \dots \circ f_{M_1}^{m_{M_1}}(f_{\tilde{i}}(\Pi(\hat{\mathbf{i}})))-f_{1}^{m_1}\circ \dots \circ f_{M_1}^{m_{M_1}}(f_{\tilde{i}}(\Pi'(\hat{\mathbf{i}})))\right|\\
&\quad\quad -|\Pi'(\mathbf{j})-\Pi(\mathbf{j})|\\
  &\quad=:L_1-L_2-L_3.
  \end{align*}
  To estimate the above expression, we need to individually estimate the values of $L_1$, $L_2$, and $L_3$. Thus, we have 
\begin{align*}
    L_1&=\left|f_{1}^{m_1}\circ \dots \circ f_{M_1}^{m_{M_1}}(f_{\tilde{i}}(\Pi'(\hat{\mathbf{i}})))-f_{1}^{m_1}\circ \dots \circ f_{r-1}^{m_{r-1}}\circ {f}_{{r}}'^{m_{r}}\circ f_{{r+1}}^{m_{r+1}}\circ \dots \circ f_{M_1}^{m_{M_1}}(f_{\tilde{i}}(\Pi'(\hat{\mathbf{i}})))\right|\\
    &=\left|{\lambda_{r}'}^{m_{r}}-\lambda_{r}^{m_{r}}\right|\prod_{\substack{i=1\\ i\neq r}}^{M_1}\lambda_{i}^{m_i}\left|f_{\tilde{i}}(\Pi'(\hat{\mathbf{i}}))\right|\geq m_r\lambda_r^{-1}\delta\prod_{i=1}^{M_1}\lambda_{i}^{m_i}\epsilon\geq \frac{N_0\delta \epsilon}{1-\epsilon} \lambda_{b_{1}^{\mathbf{i}}}.
\end{align*}
On the other hand,
\begin{align*}
    L_2&=\left|f_{1}^{m_1}\circ \dots \circ f_{M_1}^{m_{M_1}}(f_{\tilde{i}}(\Pi(\hat{\mathbf{i}})))-f_{1}^{m_1}\circ \dots \circ f_{M_1}^{m_{M_1}}(f_{\tilde{i}}(\Pi'(\hat{\mathbf{i}})))\right|\\
    &\leq \lambda_{b_{1}^{\mathbf{i}}}|\Pi(\hat{\mathbf{i}})-\Pi'(\hat{\mathbf{i}})|\leq \frac{\delta}{\epsilon} \lambda_{b_{1}^{\mathbf{i}}},
\end{align*}
where we used \eqref{eq:disp}. Moreover, 
\begin{align*}
    L_3&=|\Pi'(\mathbf{j})-\Pi(\mathbf{j})|=|f_{1}^{n_1}\circ\dots\circ f_{M_1}^{n_{M_1}}(f_{\tilde{j}}(\Pi(\hat{\mathbf{j}})))-f_{1}^{n_1}\circ\dots\circ f_{M_1}^{n_{M_1}}(f_{\tilde{j}}(\Pi'(\hat{\mathbf{j}})))|\\
    &=\lambda_{b_{1}^{\mathbf{j}}}|\lambda_{\tilde{j}}||\Pi(\hat{\mathbf{j}})-\Pi'(\hat{\mathbf{j}})|\leq \lambda_{b_{1}^{\mathbf{j}}}(1-\epsilon)\frac{\delta}{\epsilon},
\end{align*}
where we used the estimate \eqref{eq:disp} again. Thus, we get 
\begin{align*}
  L_1-L_2\geq \bigg(\frac{N_{0}\epsilon}{1-\epsilon}-\frac{1}{\epsilon}\bigg)\lambda_{b_{1}^{\mathbf{i}}}\delta\geq \frac{\epsilon\delta}{2A}\bigg(\frac{N_{0}\epsilon}{1-\epsilon}-\frac{1}{\epsilon}\bigg) \lambda_{b_{1}^{\mathbf{j}}}.
\end{align*}
Thus, by using the above estimates, we get 
\begin{align*}
    L_{1}-L_{2}-L_{3}\geq \bigg(\frac{N_{0}\epsilon^{3}-(1-\epsilon)(2A+\epsilon)}{2A(1-\epsilon)\epsilon}\bigg)\lambda_{b_{1}^{\mathbf{j}}} \delta\geq \delta\bigg(\frac{N_{0}\epsilon^{3}-(1-\epsilon)(2A+\epsilon)}{2A(1-\epsilon)\epsilon}\bigg)\epsilon^{|b_{1}^{\mathbf{j}}|}.
\end{align*}
Choosing $C=\frac{N_{0}\epsilon^{3}-(1-\epsilon)(2A+\epsilon)}{2A(1-\epsilon)\epsilon},$ and since $N_{0}=\lceil \frac{(1-\epsilon)(2A+\epsilon)}{{\epsilon}^3} \rceil+1,$ we have $C>0.$ This implies that 
$$|(\Pi(\mathbf{i})-\Pi(\mathbf{j}))-(\Pi'(\mathbf{i})-\Pi'(\mathbf{j}))|\geq C{\epsilon}^{|b_{1}^{\mathbf{j}}|}\delta.$$ This also implies that
$$\bigg{|}\frac{\partial(\Pi(\mathbf{i})-\Pi(\mathbf{j}))}{\partial{\lambda_{r}}}\bigg{|}\geq C{\epsilon}^{|b_{1}^{\mathbf{j}}|}\geq C{\epsilon}^{\max\{|b_{1}^{\mathbf{i}}|,|b_{1}^{\mathbf{j}}|\}}.$$
This completes the proof. 
\end{proof}

\subsection{Block structure}\label{sec:block} Now, we give a more detailed description of the coding of IFS $\mathcal{S}$ with common fixed point structure defined in \eqref{eq:mainIFS}. Let us recall here the definition. Let $N\geq2$ be an integer, and for every $1\leq i\leq N$, let $n_i\geq1$ be an integer. Let
	\begin{equation}\label{eq:goalIFS}
    \mathcal{S}=\{f_{i,j}=\lambda_{i,j}x+t_{i}(1-\lambda_{i,j}): 1\leq i\leq N~\&~1\leq j\leq n_i\},
    \end{equation}
    where $\lambda_{i,j}\in (0,1)$ and $t_i\in\mathbb{R}$. Without loss of generality, we assume that $0< t_1 < t_2 <\dots< t_N.$  
    
    Let $I=\{(i,j):1\leq i\leq N,\ 1\leq j\leq n_i\}$ and let $\Sigma=I^\N$ be the symbolic space corresponding to IFS $\mathcal{S}$. Denote $\Sigma_n$ the set $n$-length finite sequences formed by the symbols of $I$, and let $\Sigma^*=\bigcup_{n=0}^\infty\Sigma_n$. Denote the length of $\mathbf{i}\in\Sigma^*$ by $|\mathbf{i}|$. We will use the convention that $|\mathbf{i}|=\infty$ for $\mathbf{i}\in\Sigma$. Denote $\sigma\colon\Sigma\to\Sigma$ the usual left-shift operator. We extend $\sigma$ to $\Sigma^*$ by $\sigma(i_1,\ldots,i_n)=(i_2,\ldots,i_n)$.
    
    For each $\mathbf{i}\in \Sigma\cup\Sigma^*$, we define the following unique block representation: For $\mathbf{i}=((i_1,j_1),(i_2,j_2),\dots)$, let $|b_1^{\mathbf{i}}|:=\max\{k\geq1:i_1=i_\ell\text{ for all }1\leq\ell\leq k\}$. If $|b_1^{\mathbf{i}}|=|\mathbf{i}|$ then let $b_1^{\mathbf{i}}:=\mathbf{i}$, otherwise let $b_1^{\mathbf{i}}:=\mathbf{i}_{|_{|b_1^{\mathbf{i}}|}}$. Then we define by induction. Suppose that $b_1^{\mathbf{i}},\dots,b_n^{\mathbf{i}}$ are defined and finite. Then let 
    $$
    |b_{n+1}^{\mathbf{i}}|:=\max\left\{k\geq1:i_{|b_1^{\mathbf{i}}|+\dots+|b_n^{\mathbf{i}}|+1}=i_{|b_1^{\mathbf{i}}|+\dots+|b_n^{\mathbf{i}}|+\ell}\text{ for all }1\leq\ell\leq k\right\}.
    $$
    If $|b_{n+1}^{\mathbf{i}}|=\left|\sigma^{|b_1^{\mathbf{i}}|+\dots+|b_n^{\mathbf{i}}|}\mathbf{i}\right|$ then let $b_{n+1}^{\mathbf{i}}:=\sigma^{|b_1^{\mathbf{i}}|+\dots+|b_n^{\mathbf{i}}|}\mathbf{i}$, and so, $\mathbf{i}=b_1^{\mathbf{i}}\dots b_n^{\mathbf{i}}b_{n+1}^{\mathbf{i}}$. Otherwise, let $b_{n+1}^{\mathbf{i}}:=(\sigma^{|b_1^{\mathbf{i}}|+\dots+|b_n^{\mathbf{i}}|}\mathbf{i})_{|_{|b_{n+1}^{\mathbf{i}}|}}$. Let $B^{\mathbf{i}}$ denote the number of block in $\mathbf{i}$.     
    For $\mathbf{i}\in \Sigma$, the maps in each block $b_{l}^{\mathbf{i}}$ correspond to the same fixed points. The maps having the same fixed point are commutative. We write $t_{b_l^{\mathbf{i}}}$ for the common fixed point of the maps in the block $b_{l}^{\mathbf{i}}$, which is also the fixed point of the map $f_{b_l^{\mathbf{i}}}$. Thus, any $\mathbf{i}\in\Sigma$ can be written uniquely as countably (but possibly finitely) many finite sequences:
\begin{equation}\label{blockrep1}
\mathbf{i}=\underbrace{(i_1,j_1)\dots (i_1,j_l)}_{b_{1}^{\mathbf{i}}}\underbrace{(i_{l+1},j_{l+1})\dots (i_{l+1},j_m)}_{b_{2}^{\mathbf{i}}} (i_{m+1},j_{m+1})\dots.
\end{equation}
Note that the maps corresponding to the symbols from the same block share the same fixed point. 

Let $\Pi: \Sigma \to [t_1,t_N]$ be the natural projection corresponding to the IFS $\mathcal{S}.$ Simple algebraic manipulations show that for every $\mathbf{i}\in\Sigma$ with $B^{\mathbf{i}}=\infty$
$$
\Pi(\mathbf{i})=t_{b_{1}^{\mathbf{i}}}+\sum_{l=1}^{\infty}\lambda_{b_{1}^{\mathbf{i}}}\lambda_{b_{2}^{\mathbf{i}}}\dots \lambda_{b_{l}^{\mathbf{i}}}\left(t_{b_{l+1}^{\mathbf{i}}}-t_{b_{l}^{\mathbf{i}}}\right),
$$
where $\lambda_{b_{l}^{\mathbf{i}}}$ represent the similarity ratio of the map corresponding to the $l^{\rm th}$ block in the infinite sequence $\mathbf{i}$. For a finite $\mathbf{i}\in\Sigma^*$, we will write $\Pi(\mathbf{i})=f_{\mathbf{i}}(0)$ with a slight abuse of notation. Observe that $\Pi$ is an analytic map of $\underline{\lambda}=(\lambda_{i,j})_{(i,j)\in I}$ and a linear map of $\underline{t}=(t_i)_{i=1}^N$ for every $\mathbf{i}\in\Sigma\cup\Sigma^*$. In particular, $\Pi$ is analytic over $\underline{\lambda}\in(-1+\epsilon,1-\epsilon)^{\sum_{i=1}^Nn_i}$, which fact will be used for later purposes.

For a finite word $\mathbf{i}\in\Sigma^*$, let $\#_{i,j}\mathbf{i}$ denote the number of symbols $(i,j)$ appearing in $\mathbf{i}$. For $\mathbf{i},\mathbf{j}\in\Sigma\cup\Sigma^*$, we say that $\mathbf{i},\mathbf{j}$ have the same block structure if $B^{\mathbf{i}}=B^{\mathbf{j}}$ and for every $1\leq l\leq B^{\mathbf{i}}$, $t_{b_l^{\mathbf{i}}}=t_{b_l^{\mathbf{j}}}$ and $\#_{i,j}b_l^{\mathbf{i}}=\#_{i,j}b_l^{\mathbf{j}}$ for every $(i,j)\in I$. Note that if $\mathbf{i},\mathbf{j}\in\Sigma^*$ have the same block structure then $f_{\mathbf{i}}\equiv f_{\mathbf{j}}$.

Now, we define a special form of the WESC adapted to common fixed-point systems.

\begin{definition}\label{DefWESCCFS} We say that IFS $\mathcal{S}$ satisfies the \texttt{Exponential Separation Condition for the Common Fixed Point System (ESC for CFS)}, if there exist $N\in \mathbb{N}$ and $b>1$ such that for every $n\geq N$ and every $\mathbf{i}, \mathbf{j}\in \Sigma_n$ with $\lambda_{\mathbf{i}}=\lambda_{\mathbf{j}}$, we have the following:
\begin{equation}\label{eq:ESCCFS}
\text{either}~\mathbf{i} , \mathbf{j}~\text{have the same block structure or}~ |\Pi({\mathbf{i}})-\Pi({\mathbf{j}})|>2^{-b n}.
\end{equation}
\end{definition}
Let us note that one can relax the definition of the exponential separation for common fixed point systems (ESC for CFS) through a subsequence of natural numbers $n$ for which \eqref{eq:ESCCFS} holds. However, it would require a slight change in the proofs and the definition of the exceptional sets.

One can see that ESC for CFS for the IFS $\mathcal{S}$ implies that  IFS $\mathcal{S}$ satisfies the WESC (see Definition \ref{WESC}). ESC for CFS provides a complete description of the overlapping cylinders via symbolic space, see Section~\ref{sec:rw}. 

We say that the finite words $\mathbf{i},\mathbf{j}\in\Sigma^*$ are disjoint if $\min\{\#_{i,j}\mathbf{i},\#_{i,j}\mathbf{j}\}=0$ for every $(i,j)\in I$. We will denote this with a slight abuse of notation by $\mathbf{j}\cap\mathbf{i}=\emptyset$.



\subsection{Two-fixed point system}\label{2fixedpointsys} Next, we discuss the ESC for CFS in the case of the two fixed-point systems. We consider the following iterated function system, having only two fixed points
\begin{equation}\label{eq:IFS2fix}
\mathcal{J}=\{f_{1,i}(x)=\lambda_{1,i} x\}_{i=1}^{n_1}\cup \{f_{2,i}(x)=\lambda_{2,i} x+(1-\lambda_{2,i})\}_{i=1}^{n_2}.
\end{equation}
Clearly, the IFS $\mathcal{J}$ is a special case of the family of IFSs defined in \eqref{eq:goalIFS}, where $t_1=0$ and $t_2=1$. Let $\Sigma$ be the symbolic space corresponding to the IFS $\mathcal{J}$ with symbols $I=\{(i,j): i\in\{1,2\}, j\in\{1,\ldots,n_i\}\}.$

Denote $L=n_1+n_2$, and $\underline{\lambda}=(\lambda_{i,j})_{(i,j)\in I}\in(0,1)^{L}$ the vector of contraction ratios.  Observe that for every $\epsilon>0$ such that $\epsilon\leq\min_{(i,j)\in I}\min\{\lambda_{i,j},1-\lambda_{i,j}\}$ then $L_{\mathcal{J}}\subseteq [0,1-\epsilon]$ and $R_{\mathcal{J}}\subseteq [\epsilon,1]$, where $L_{\mathcal{J}}:=\bigcup_{i=1}^{n_1}f_{1,i}([0,1])$ and $R_{\mathcal{J}}:=\bigcup_{i=1}^{n_2}f_{2,i}([0,1]).$ 

Let $\Pi$ be the natural projection corresponding to the IFS $\mathcal{J}$. For $\mathbf{i},\mathbf{j}\in \Sigma$, we define $$
\Delta_{\mathbf{i},\mathbf{j}}(\underline{\lambda})=\Pi(\mathbf{i})-\Pi(\mathbf{j}).
$$ 

Let us define the following set of pairs:
\begin{align}\label{firstsetB}
\mathcal{B}:=\bigg\{(\mathbf{i},\mathbf{j})\in \Sigma\times \Sigma: b_1^{\mathbf{i}}\cap b_1^{\mathbf{j}}=\emptyset~\&~b_1^{\mathbf{i}}\neq\mathbf{i}~\&~b_1^{\mathbf{j}}\neq\mathbf{j}\bigg\}.
\end{align}
We divide the set $\mathcal{B}$ further:
\begin{equation}\label{eqBdiv1}
\mathcal{B}_{1}:=\bigg\{(\mathbf{i},\mathbf{j})\in \mathcal{B}: t_{b_{1}^{\mathbf{i}}}\ne t_{b_{1}^{\mathbf{j}}} \bigg\}\text{ and }
\mathcal{B}_2:=\mathcal{B}\setminus\mathcal{B}_1=\left\{(\mathbf{i},\mathbf{j})\in \mathcal{B}: t_{b_1^{\mathbf{i}}}=t_{b_1^{\mathbf{j}}}\right\}.
\end{equation}
Now, set $\tilde{N_{0}}:=\lceil \frac{(1-\epsilon)(2+\epsilon)}{\epsilon^3} \rceil+1$, and divide $\mathcal{B}_2$ further: 
\begin{equation}\label{eqBdiv2}
  \mathcal{B}_{3}=\bigg\{(\mathbf{i},\mathbf{j})\in \mathcal{B}_2:~\max_{k\in I}\left\{\max\left\{\#_kb_{1}^{\mathbf{i}},\#_kb_{1}^{\mathbf{j}}\right\}\right\}\leq \tilde{N_{0}}\bigg\}\text{ and }\mathcal{B}_4=\mathcal{B}_2\setminus\mathcal{B}_3.
\end{equation}
Observe that $\mathcal{B}_3$ is a compact subset of $\Sigma\times\Sigma$.

\begin{lemma}\label{ExponentialA1b}
    Let $\epsilon>0$ be arbitrary. Then there exists a constant $C>0$ such that for every $\underline{\lambda}\in[\epsilon,1-\epsilon]^{L}$ and for every $(\mathbf{i},\mathbf{j})\in \mathcal{B}_4$
    $$
    \max\left\{|\Delta_{\mathbf{i},\mathbf{j}}(\underline{\lambda})|,\bigg{|}\frac{\partial\Delta_{\mathbf{i},\mathbf{j}}}{\partial \lambda_{k}}(\underline{\lambda})\bigg{|}\right\}\geq C{\epsilon}^{3\max\{|b_{1}^{\mathbf{i}}|,|b_{1}^{\mathbf{j}}|\}},$$ where $k$ is such that $\max\{\#_kb_{1}^{\mathbf{i}},\#_kb_{1}^{\mathbf{j}}\}>\tilde{N_{0}}$.
\end{lemma}

\begin{proof} Let $(\mathbf{i},\mathbf{j})\in \mathcal{B}_4$. Without loss of generality, we may assume that $t_{b_1^{\mathbf{i}}}=t_{b_1^{\mathbf{j}}}=0$. The case $t_{b_1^{\mathbf{i}}}=t_{b_1^{\mathbf{j}}}=1$ can be handled similarly by considering the conjugated IFS with conjugation map $x\mapsto 1-x$.

First, suppose that $\frac{\lambda_{b_{1}^{\mathbf{i}}}}{\lambda_{b_{1}^{\mathbf{j}}}}\notin \bigg(\dfrac{{\epsilon}}{2},\dfrac{2}{{\epsilon}}\bigg)$. Then, by Lemma \ref{GenExponentialA1}, we get the
$$
|\Delta_{\mathbf{i},\mathbf{j}}(\underline{\lambda})|\geq {\epsilon}^{3\max\{|b_{1}^{\mathbf{i}}|,|b_{1}^{\mathbf{j}}|\}}.
$$

Now, we assume that $\frac{\lambda_{b_{1}^{\mathbf{i}}}}{\lambda_{b_{1}^{\mathbf{j}}}}\in \bigg(\dfrac{{\epsilon}}{2},\dfrac{2}{{\epsilon}}\bigg)$. Then by Lemma \ref{GenFirstorderTrans}, we get that 
$$
\bigg{|}\frac{\partial\Delta_{\mathbf{i},\mathbf{j}}}{\partial \lambda_{k}}(\underline{\lambda})\bigg{|}\geq C{\epsilon}^{\max\{|b_{1}^{\mathbf{i}}|,|b_{1}^{\mathbf{j}}|\}}
$$
for some uniform constant $C>0$.
\end{proof}

\begin{lemma}\label{Trans3} Let $\epsilon>0$ be arbitrary. Then there exist $p\geq0$ and $\tilde{C}>0$ such that for every  $(\mathbf{i},\mathbf{j})\in \mathcal{B}_3$ and for all $\underline{\lambda}\in [\epsilon, 1-\epsilon]^{L}$, there exists $(m_{i,j})_{(i,j)\in I}\in\N^L$ such that $m=\sum_{(i,j)\in I}m_{i,j}\leq p$ and
$$\left|\frac{\partial^m\Delta_{\mathbf{i},\mathbf{j}}}{\prod_{(i,j)\in I}\partial^{m_{i,j}}\lambda_{i,j}}(\underline{\lambda})\right|>\tilde{C}.$$ 
\end{lemma}

We use the idea of Hochman \cite[Proposition 6.24]{Hochman2022} to prove this lemma.

\begin{proof}
   Let $(\mathbf{i},\mathbf{j})\in \mathcal{B}_3$. Without loss of generality, we may assume that  $t_{b_1^{\mathbf{i}}}=t_{b_1^{\mathbf{j}}}=0$. The case $t_{b_1^{\mathbf{i}}}=t_{b_1^{\mathbf{j}}}=1$ can be handled similarly by considering the conjugated IFS with conjugation map $x\mapsto 1-x$.
   
   First, we show that $\Delta_{\mathbf{i},\mathbf{j}}(\underline{\lambda})\not\equiv 0$ over $\underline{\lambda}\in [\epsilon, 1-\epsilon]^{L}$. We have 
    \begin{align*}
   \Delta_{\mathbf{i},\mathbf{j}}(\underline{\lambda})&=\Pi(\mathbf{i})-\Pi(\mathbf{j})\\&=\sum_{l=1}^{\infty}(-1)^{l+1}\lambda_{b_{1}^{\mathbf{i}}}\lambda_{b_{2}^{\mathbf{i}}}\dots \lambda_{b_{l}^{\mathbf{i}}}-\sum_{l=1}^{\infty}(-1)^{l+1}\lambda_{b_{1}^{\mathbf{j}}}\lambda_{b_{2}^{\mathbf{j}}}\dots \lambda_{b_{l}^{\mathbf{j}}}\\&= \lambda_{b_{1}^{\mathbf{i}}}-\lambda_{b_{1}^{\mathbf{j}}}+ \sum_{l=2}^{\infty}(-1)^{l+1}\lambda_{b_{1}^{\mathbf{i}}}\lambda_{b_{2}^{\mathbf{i}}}\dots \lambda_{b_{l}^{\mathbf{i}}}-\sum_{l=2}^{\infty}(-1)^{l+1}\lambda_{b_{1}^{\mathbf{j}}}\lambda_{b_{2}^{\mathbf{j}}}\dots \lambda_{b_{l}^{\mathbf{j}}}.   \end{align*} 
   
   Since $b_{1}^{\mathbf{i}}\cap b_{1}^{\mathbf{j}}=\emptyset$, we get
   \begin{align*}
 \frac{{\partial}^{|b_{1}^{\mathbf{i}}|} \Delta_{\mathbf{i},\mathbf{j}}}{\prod_{(i,j)\in I}{\partial}^{\#_{i,j}b_{1}^{\mathbf{i}}} \lambda_{i,j}}(\underline{\lambda})&= \prod_{(i,j)\in I}(\#_{i,j}b_{1}^{\mathbf{i}})!+ \sum_{l=2}\frac{{\partial}^{|b_{1}^{\mathbf{i}}|} ((-1)^{l+1}\lambda_{b_{1}^{\mathbf{i}}}\lambda_{b_{2}^{\mathbf{i}}}\dots \lambda_{b_{l}^{\mathbf{i}}})}{\prod_{(i,j)\in I}{\partial}^{\#_{i,j}b_{1}^{\mathbf{i}}} \lambda_{i,j}}\\
 &\quad-\sum_{l=2}\frac{{\partial}^{|b_{1}^{\mathbf{i}}|} ((-1)^{l+1}\lambda_{b_{1}^{\mathbf{i}}}\lambda_{b_{2}^{\mathbf{j}}}\dots \lambda_{b_{l}^{\mathbf{j}}})}{\prod_{(i,j)\in I}{\partial}^{\#_{i,j}b_{1}^{\mathbf{i}}} \lambda_{i,j}}.      
   \end{align*}
This implies that $\frac{{\partial}^{|b_{1}^{\mathbf{i}}|} \Delta_{\mathbf{i},\mathbf{j}}}{\prod_{(i,j)\in I}{\partial}^{\#_{i,j}b_{1}^{\mathbf{i}}} \lambda_{i,j}}(\underline{0})=\prod_{(i,j)\in I}(\#_{i,j}b_{1}^{\mathbf{i}})!$.  Since the function $\Delta_{\mathbf{i},\mathbf{j}}(\underline{\lambda})$ is analytic on $(-1+\frac{\epsilon}{2}, 1-\frac{\epsilon}{2})^{L}$, we have $\Delta_{\mathbf{i},\mathbf{j}}(\underline{\lambda}) \not\equiv 0$ over $\underline{\lambda}\in [\epsilon, 1-\epsilon]^{L}$. 

Let's continue the proof using an argument by contradiction. Assume that $\forall~n\geq 1$, $\exists~ \underline{\lambda}_n\in [\epsilon, 1-\epsilon]^{L}$ and $\exists~ (\mathbf{i}^{n},\mathbf{j}^{n})\in \mathcal{B}_3$ such that   
$$
\bigg{|} \frac{\partial^m\Delta_{\mathbf{i}^{n},\mathbf{j}^{n}}}{\prod_{(i,j)\in I}\partial^{m_{i,j}}\lambda_{i,j}}(\underline{\lambda}_n)\bigg{|}<\frac{1}{n}\quad \forall~ (m_{i,j})_{(i,j)\in I}\in\N^{L}\text{ with }m=\sum_{(i,j)\in I}m_{i,j}\leq n.
$$ 
Due to the compactness of $\mathcal{B}_{3},$ there exists a subsequence $\{n_l\}_{l\geq 1}$, $\underline{\lambda}_{0}\in [\epsilon, 1-\epsilon]^{L}$ and $(\mathbf{i},\mathbf{j})\in \mathcal{B}_3$ such that 
$$\lim_{l\to \infty}\underline{\lambda}_{n_l}=\underline{\lambda}_{0},\quad \lim_{l\to \infty}(\mathbf{i}^{n_l},\mathbf{j}^{n_l} )=(\mathbf{i},\mathbf{j}).$$
This also implies that 
$$
\frac{\partial^m\Delta_{\mathbf{i}^{n_l},\mathbf{j}^{n_l}}}{\prod_{(i,j)\in I}\partial^{m_{i,j}}\lambda_{i,j}} \to \frac{\partial^m\Delta_{\mathbf{i},\mathbf{j}}}{\prod_{(i,j)\in I}\partial^{m_{i,j}}\lambda_{i,j}}  \quad \text{uniformly on }[\epsilon, 1-\epsilon]^{L}
$$
for every $(m_{i,j})_{(i,j)\in I}\in\N^{L}$. Thus, 
$$\frac{\partial^m\Delta_{\mathbf{i},\mathbf{j}}}{\prod_{(i,j)\in I}\partial^{m_{i,j}}\lambda_{i,j}}(\underline{\lambda}_{0})=0\quad \forall~ (m_{i,j})_{(i,j)\in I}\in\N^{L}$$ Since $\Delta_{\mathbf{i},\mathbf{j}}(\underline{\lambda})$ is analytic function, we have $\Delta_{\mathbf{i},\mathbf{j}}(\underline{\lambda})\equiv 0$, which is a contradiction.
\end{proof}

\begin{lemma}\label{nonzeroA0}
   Let $\epsilon>0$ be arbitrary. Then there exist $p\geq0$ and $\tilde{C}>0$ such that for every  $(\mathbf{i},\mathbf{j})\in \mathcal{B}_1$ and for all $\underline{\lambda}\in [\epsilon, 1-\epsilon]^{L}$, there exists $(m_{i,j})_{(i,j)\in I}\in\N^L$ such that $m=\sum_{(i,j)\in I}m_{i,j}\leq p$ and
$$\left|\frac{\partial^m\Delta_{\mathbf{i},\mathbf{j}}}{\prod_{(i,j)\in I}\partial^{m_{i,j}}\lambda_{i,j}}(\underline{\lambda})\right|>\tilde{C}.$$ 
\end{lemma}

\begin{proof}
Actually, we claim that the statement of the lemma holds for the greater set $\overline{\mathcal{B}_1}=\bigg\{(\mathbf{i},\mathbf{j})\in \Sigma\times\Sigma: t_{b_{1}^{\mathbf{i}}}\ne t_{b_{1}^{\mathbf{j}}} \bigg\}$, and clearly, $\mathcal{B}_1\subseteq\overline{\mathcal{B}_1}$. Observe that the set $\overline{\mathcal{B}_1}$ is a compact subset of $\Sigma\times\Sigma$, and so, the proof can be similarly completed to the proof of Lemma~\ref{Trans3}. We leave the details for the reader.
\end{proof}

For $n\geq1$, let 
$$
\mathcal{B}^{(n)}=\left\{(\mathbf{i},\mathbf{j})\in\Sigma_n\times\Sigma_n: b_1^{\mathbf{i}}\cap b_1^{\mathbf{j}}=\emptyset~\&~b_1^{\mathbf{i}}\neq\mathbf{i}~\&~b_1^{\mathbf{j}}\neq\mathbf{j}\right\},
$$
and divide it further into $\mathcal{B}^{(n)}_1, \mathcal{B}^{(n)}_2, \mathcal{B}^{(n)}_3$ and $\mathcal{B}^{(n)}_4$ as in \eqref{eqBdiv1} and \eqref{eqBdiv2} by
$$
\mathcal{B}^{(n)}_k=\{(\mathbf{i},\mathbf{j})\in\mathcal{B}^{(n)}:~(\mathbf{i}(1,1)\cdots,\mathbf{j}(1,1)\cdots)\in\mathcal{B}_k\}.
$$
Define 
\begin{equation}\label{eq:E1}
E_\epsilon=\bigcap_{H\geq 1}\bigcap_{\tilde{N}\geq1}\bigcup_{n\geq \tilde{N}}\bigcup_{(\mathbf{i},\mathbf{j})\in\mathcal{B}^{(n)}}\bigg\{\underline{\lambda}\in[\epsilon,1-\epsilon]^L: |\Delta_{\mathbf{i},\mathbf{j}}(\underline{\lambda})|< \frac{1}{H^n}\bigg\}.
\end{equation}

\begin{proposition}\label{ExponentialE1}
    Let $E_\epsilon\subset [\epsilon, 1-\epsilon]^L$ be the set defined as above. Then, $\dim_{H}(E_\epsilon)\leq L-1.$ 
\end{proposition}

\begin{proof} 
By combining Lemma~\ref{ExponentialA1b} and \cite[Lemma~6.26]{Hochman2022}, then for every $(\mathbf{i},\mathbf{j})\in\mathcal{B}^{(n)}_4$  with $n\geq1$ the set
$$
E_{\epsilon,H}^{(\mathbf{i},\mathbf{j})}=\bigg\{\underline{\lambda}\in[\epsilon,1-\epsilon]^L: |\Delta_{\mathbf{i},\mathbf{j}}(\underline{\lambda})|< \frac{1}{H^n}\bigg\}
$$
can be covered by $C'(C\epsilon^{3n}H^n)^{L-1}$-many balls of radius $H^{-n}C^{-1}\epsilon^{-3n}$ for $\epsilon^{-3}<H$, where $C'>0$ an universal constant. 

By \cite[Proposition 6.28]{Hochman2022}, Lemma~\ref{Trans3} and Lemma~\ref{nonzeroA0}, we get that for every $(\mathbf{i},\mathbf{j})\in\mathcal{B}^{(n)}_1\cup\mathcal{B}^{(n)}_3$, the set $E_{\epsilon,H}^{(\mathbf{i},\mathbf{j})}$ can be covered by $C'H^{n(L-1)/2^p}$-many balls of radius $\tilde{C}^{-1}H^{-n/2^p}$, for sufficiently large $H$. Choose $H$ so large that, $H^{-n}C^{-1}\epsilon^{-3n}<\tilde{C}^{-1}H^{-n/2^p}$ for every $n\in\mathbb{N}$. 
Let
$$
E_{\epsilon,H}=\bigcap_{\tilde{N}\geq1}\bigcup_{n\geq \tilde{N}}\bigcup_{(\mathbf{i},\mathbf{j})\in\mathcal{B}^{(n)}}\bigg\{\underline{\lambda}\in[\epsilon,1-\epsilon]^L: |\Delta_{\mathbf{i},\mathbf{j}}(\underline{\lambda})|< \frac{1}{H^n}\bigg\}.
$$
Thus, we have for every $\tilde{N}\geq1$
\begin{align*}
  \mathcal{H}_{\tilde{C}^{-1}H^{-\tilde{N}/2^p}}^{s}(E_{\epsilon,H})&\leq \sum_{n\geq \tilde{N}} L^{2n}C'(C\epsilon^{3n}H^n)^{L-1} \bigg(H^{-n}C^{-1}\epsilon^{-3n}\bigg)^s+L^{2n}C'H^{n(L-1)/2^p}\bigg(\tilde{C}^{-1}H^{-n/2^p}\bigg)^s\\
  &=C'C^{L-1-s}\sum_{n\geq \tilde{N}} \left(L^{2}(\epsilon^{3}H)^{L-1-s}\right)^{n}+C'\tilde{C}^{-s}\sum_{n\geq \tilde{N}}\left(L^2H^{\frac{L-1-s}{2^p}}\right)^n,
\end{align*}
which is finite if $s>\max\{\frac{2\log{L}}{\log(H\epsilon^3)}+L-1,\frac{2^{p+1}\log{L}}{\log(H)}+L-1\}$. This implies that $\dim_{H}(E_{\epsilon,H})\leq \max\{\frac{2\log{L}}{\log(H\epsilon^3)},\frac{2^{p+1}\log{L}}{\log(H)}\}+L-1.$ Since $E_{\epsilon}\subseteq \bigcap_{H\geq 1}E_{\epsilon,H}$, we have $\dim_{H}E_{\epsilon}\leq L-1$. This completes the proof.
\end{proof}

We define a set as follows:
\begin{equation}\label{eq:G1}
G_\epsilon=\bigcup_{n=1}^\infty\bigcup_{m=1}^\infty\bigcup_{\substack{(\mathbf{i},\mathbf{j})\in\Sigma_n\times\Sigma_m\\ \mathbf{i}\cap\mathbf{j}=\emptyset}}\{\underline{\lambda}\in[\epsilon,1-\epsilon]^L:\lambda_{\mathbf{i}}=\lambda_{\mathbf{j}}\}\text{ and }G=\bigcup_{n=1}^\infty G_{1/n}.
\end{equation}

\begin{lemma}\label{Descrofmoreblock}
  Let $\varepsilon>0$ be arbitrary and let $G_\epsilon$ be as above. Then $\dim_HG_\epsilon\leq L-1$. Moreover, for every $\underline{\lambda}\in[\epsilon,1-\epsilon]^L\setminus G_\epsilon$ and $(\mathbf{i},\mathbf{j})\in \Sigma^*\times\Sigma^*$, 
  $$\text{either }\#_{i,j}(\mathbf{i})=\#_{i,j}(\mathbf{j}) ~\forall~(i,j)\in I\text{ or } \lambda_{\mathbf{i}}\neq\lambda_{\mathbf{j}}.$$  
\end{lemma}
\begin{proof}
 It is clear that for every $n,m\in\N$ and $\mathbf{i}\in\Sigma_n,\mathbf{j}\in\Sigma_m$ with $\mathbf{i}\cap\mathbf{j}=\emptyset$, the set $\{\underline{\lambda}\in[\epsilon,1-\epsilon]^L:\lambda_{\mathbf{i}}=\lambda_{\mathbf{j}}\}$ forms a smooth $(L-1)$-dimensional manifold, and so, $\dim_HG_{\epsilon}\leq L-1$.

 On the other hand, for every $(\mathbf{i},\mathbf{j})\in \Sigma^*\times\Sigma^*$ there are finite words $\mathbf{h},\mathbf{i}',\mathbf{j}'\in\Sigma^*$ such that 
 $$
 \#_{i,j}\mathbf{h}=\min\{\#_{i,j}\mathbf{i},\#_{i,j}\mathbf{j}\},~\#_{i,j}\mathbf{i}'=\#_{i,j}\mathbf{i}-\#_{i,j}\mathbf{h}\text{ and }\#_{i,j}\mathbf{j}'=\#_{i,j}\mathbf{j}-\#_{i,j}\mathbf{h},
 $$
 moreover,  $\lambda_{\mathbf{i}}=\lambda_{\mathbf{h}}\lambda_{\mathbf{i}'}$, $\lambda_{\mathbf{j}}=\lambda_{\mathbf{h}}\lambda_{\mathbf{j}'}$. Clearly, $\mathbf{i}'\cap\mathbf{j}'=\emptyset$ and $\lambda_{\mathbf{i}}-\lambda_{\mathbf{j}}=\lambda_{\mathbf{h}}\left(\lambda_{\mathbf{i}'}-\lambda_{\mathbf{j}'}\right)\neq 0$ by $\underline{\lambda}\in[\epsilon,1-\epsilon]^L\setminus G_\epsilon$. 
\end{proof}

\begin{proposition}\label{Exceptionfor2fixedpoint}
 There exists a set $\mathbf{E}\subset (0,1)^L$ such that $\dim_{H}\mathbf{E}\leq L-1$ and the IFS $\mathcal{J}$ defined satisfies ESC for CFS for every parameters $\underline{\lambda}\in (0, 1)^{L}\setminus\mathbf{E}$. 
\end{proposition}	
\begin{proof}
We define the set $\mathbf{E}=G\cup\bigcup_{n=1}^\infty E_{1/n}$, where $E_{1/n}$ is defined in \eqref{eq:E1} and $G$ is defined in \eqref{eq:G1}. By Proposition~\ref{ExponentialE1} and Lemma~\ref{Descrofmoreblock}, $\dim_H\mathbf{E}\leq L-1$. Now, we show that the IFS $\mathcal{J}$ satisfies ESC for CFS for every parameter $\underline{\lambda}\in (0,1)^L\setminus\mathbf{E}$. 

Let $\epsilon>0$ and $\underline{\lambda}\in [\epsilon, 1-\epsilon]^{L}\setminus\mathbf{E}$ be arbitrary but fixed. Let $\mathbf{i},\mathbf{j}\in\Sigma_n$ be such that $\lambda_{\mathbf{i}}=\lambda_{\mathbf{j}}$, but $\mathbf{i}$ and $\mathbf{j}$ have different block structure. Since $\mathbf{i}$ and $\mathbf{j}$ have different block structure, there exists $\mathbf{h}, \mathbf{i}', \mathbf{j}'\in\Sigma^*$ such that 
$$
f_{\mathbf{i}}\equiv f_{\mathbf{h}}\circ f_{\mathbf{i}'},\ f_{\mathbf{j}}\equiv f_{\mathbf{h}}\circ f_{\mathbf{j}'}\text{ and }b_1^{\mathbf{i}'}\cap b_1^{\mathbf{j}'}=\emptyset.
$$
Thus, $\lambda_{\mathbf{i}'}=\lambda_{\mathbf{j}'}$. 

If $b_1^{\mathbf{i}'}=\mathbf{i}'$ and $b_1^{\mathbf{j}'}=\mathbf{j}'$, then $\lambda_{\mathbf{i}'}\neq\lambda_{\mathbf{j}'}$ by $\underline{\lambda}\notin G$, which is a contradiction. If $b_1^{\mathbf{i}'}=\mathbf{i}'$ and $b_1^{\mathbf{j}'}\neq\mathbf{j}'$, then there exists $i\in\{1,2\}$ such that $\#_{i,j}\mathbf{j}'\geq1>0=\#_{i,j}\mathbf{i}'$ for some $j\in\{1,\ldots,n_i\}$, and hence, again $\lambda_{\mathbf{i}'}\neq\lambda_{\mathbf{j}'}$ by $\underline{\lambda}\notin G$, which is again a contradiction.

Hence, $b_1^{\mathbf{i}'}\neq\mathbf{i}'$ and $b_1^{\mathbf{j}'}\neq\mathbf{j}'$. Since $\underline{\lambda}\notin E_\epsilon$, there exists $b>1$ independent of $n\in\N$ and $\mathbf{i},\mathbf{j}\in\Sigma_n$ such that
$$
|\Delta_{\mathbf{i}',\mathbf{j}'}(\underline{\lambda})|>2^{-b|\mathbf{i}'|}\geq 2^{-bn}.
$$
Hence, 
$$
|\Delta_{\mathbf{i},\mathbf{j}}(\underline{\lambda})|=\lambda_{\mathbf{h}}|\Delta_{\mathbf{i}',\mathbf{j}'}(\underline{\lambda})|>\epsilon^{n}2^{-bn},
$$
which had to be proven.
\end{proof}	

\subsection{General fixed point system}\label{Genaralcaseextend}
Now, let us turn our attention to the IFS 
$$
\mathcal{S}_{\underline{\lambda},\underline{t}}=\{f_{i,j}(x)=\lambda_{i,j}x+t_i(1-\lambda_{i,j}):i\in\{1,\ldots,N\},\ j\in\{1,\ldots,n_i\}\}
$$
defined in \eqref{eq:goalIFS}, and denote its natural projection by $\Pi$, as usual. For every $i\in\{1,\ldots,N\}$, let us define the IFS
$$\mathcal{S}_{i}=\{\lambda_{i,j}x+(1-\lambda_{i,j})\}_{j=1}^{n_i}\bigcup \{\lambda_{k,j}x:k\in\{1,\ldots,N\}\setminus\{i\}, j\in\{1,\ldots,n_k\}\},$$
and let $\Pi_{i}$ be the natural projection for the IFS $\mathcal{S}_{i}$. Then clearly,
\begin{align}\label{connectionformula}
\Pi(\mathbf{i})=\sum_{i=1}^Nt_i\Pi_i(\mathbf{i}).  
\end{align}
Let us introduce the following vector:
$$\hat{\Pi}(\mathbf{i})=(\Pi_{1}(\mathbf{i}),\dots,\Pi_{N}(\mathbf{i})).$$
Then by writing $\underline{t}=(t_1,\ldots,t_N)$, we get $\Pi(\mathbf{i})=\langle\underline{t},\hat{\Pi}(\mathbf{i})\rangle$, where $\langle\cdot,\cdot\rangle$ denotes the usual scalar product on $\R^N$ and let $\|\cdot\|$ be the usual Euclidean distance. Recall the notation $L=\sum_{i=1}^Nn_i$.

By Proposition \ref{Exceptionfor2fixedpoint}, for every $i=1,\ldots,N$ there exists a set $\mathbf{E}_{i}\subset (0,1)^L$ such that $\dim_{H}\mathbf{E}_{i}\leq L-1$ such that the IFS $\mathcal{S}_{i}$ satisfies ESC for CFS for all parameters $\underline{\lambda}\in (0,1)^L\setminus\mathbf{E}_{i}$.
We define a set by 
\begin{equation}\label{eq:E}
\mathbf{E}=\bigcup_{i=1}^{N}\mathbf{E}_{i}.
\end{equation}
Clearly, $\dim_{H}(\mathbf{E})\leq L-1$. Let $\Sigma_{n}$ be the set of all finite sequences of length $n$ with symbols from the set $I.$

\begin{lemma}\label{Desforgeneralsystem}
    Let $\underline{\lambda}\in (0,1)^L\setminus\mathbf{E}$. Then there exists $b>0$ such that for every $n\in\N$ and every $\mathbf{i},\mathbf{j}\in \Sigma_{n}$ with $\lambda_{\mathbf{i}}=\lambda_{\mathbf{j}}$, we have 
    $$\text{ either }\mathbf{i}~\text{and}~\mathbf{j}~\text{have the same block structure with respect to $\mathcal{S}$ or}~\|\hat{\Pi}(\mathbf{i})-\hat{\Pi}(\mathbf{j})\|>2^{-bn}.$$    
\end{lemma}
\begin{proof}
Since $\underline{\lambda}\in (0,1)^L\setminus\mathbf{E}$, for every $i\in\{1,\ldots,N\}$ there exists $b_i=b_i(\underline{\lambda})>0$ such that for every $\mathbf{i}$ and $\mathbf{j}$ with $\lambda_{\mathbf{i}}=\lambda_{\mathbf{j}}$ but with different block structure with respect to $\mathcal{S}_i$, we have $|\Pi_i(\mathbf{i})-\Pi_i(\mathbf{j})|>2^{-b_in}$ by Proposition \ref{Exceptionfor2fixedpoint}. Let $b:=\max\{b_1,\ldots,b_N\}$.

    Let $\mathbf{i},\mathbf{j}\in \Sigma_{n}$ be such that $\lambda_{\mathbf{i}}=\lambda_{\mathbf{j}}$ and $\mathbf{i}~\text{and}~\mathbf{j}$ does not have the same block structure with respect to $\mathcal{S}$. Let $l\geq1$ be the smallest $l$ such that the blocks $b_l^{\mathbf{i}}$ and $b_l^{\mathbf{j}}$ differ significantly. That is,
     $$
     l=\min\{k\geq1:~\exists(i,j)\in I\text{ s. t. }\#_{i,j}b_k^{\mathbf{i}}\neq\#_{i,j}b_k^{\mathbf{j}}\}.
     $$
     Let $i$ be such that $t_{b_l^{\mathbf{i}}}=t_i$. Then $\mathbf{i},\mathbf{j}\in \Sigma_{n}$ have different block structure with respect to the IFS $\mathcal{S}_i$, and so, $\|\hat{\Pi}(\mathbf{i})-\hat{\Pi}(\mathbf{j})\|\geq|\Pi_i(\mathbf{i})-\Pi_i(\mathbf{j})|>2^{-bn}$.
\end{proof}

\begin{proposition}\label{ESC for gerenalsystem}
Let $\mathbf{E}\subset (0,1)^L$ as in \eqref{eq:E}. Then for every $\underline{\lambda}\in(0,1)^L\setminus\mathbf{E}$ there exists $\mathbf{F}\subset \R^N$ such that $\dim_{H}\mathbf{F}\leq N-1$ and the IFS $\mathcal{S}$ satisfies ESC for CFS for all parameters $\underline{\mathbf{t}}\in \R^N\setminus \mathbf{F}.$     
\end{proposition}
\begin{proof}
Let $\mathbf{i},\mathbf{j}\in\Sigma_n$ be such that $\mathbf{i}$ and $\mathbf{j}$ have different block structure. By Lemma~\ref{Desforgeneralsystem}, for every $\underline{\lambda}\in(0,1)^L\setminus\mathbf{E}$ there exists $b\in\N$ (independent of the choice of $\mathbf{i},\mathbf{j}$) such that $\|\hat{\Pi}(\mathbf{i})-\hat{\Pi}(\mathbf{j})\|\geq 2^{-bn}$. For simplicity, let 
$$
\mathcal{A}_n=\{(\mathbf{i},\mathbf{j})\in \Sigma_{n}\times \Sigma_{n}: \mathbf{i}~\text{and}~\mathbf{i}~ \text{have not same block structure and}~ \lambda_{\mathbf{i}}=\lambda_{\mathbf{j}}\}.
$$

For $\underline{t}\in\R^N$ and $(\mathbf{i},\mathbf{j})\in\mathcal{A}_n$, let $\theta(\underline{t},\mathbf{i},\mathbf{j})$ be the angle between the vectors $\underline{t}$ and $\hat{\Pi}(\mathbf{i})-\hat{\Pi}(\mathbf{j})$. The exceptional set for the fixed-point parameters is $\mathbf{F}=\bigcup_{m=1}^\infty\bigcup_{M=1}^\infty\bigcap_{H=2^{b}}^\infty\mathbf{F}_{m,M,H}$, where
$$
\mathbf{F}_{m,M,H}=\bigcap_{\tilde{N}=1}^{\infty}\bigcup_{n=\tilde{N}}^{\infty}\bigcup_{(\mathbf{i},\mathbf{j})\in\mathcal{A}_n}\bigg\{\underline{t}: 1/m\leq\|\underline{t}\|\leq M~\&~|\cos(\theta(\underline{t},\mathbf{i},\mathbf{j}))|\leq m H^{-n}2^{bn}\bigg\}.
$$
Clearly, for every $\underline{t}\in\R^N\setminus\mathbf{F}$ there exist $H,\tilde{N}\in\N$ such that for every $n\geq\tilde{N}$ and $(\mathbf{i},\mathbf{j})\in\mathcal{A}_n$
$$
|\Pi(\mathbf{i})-\Pi(\mathbf{j})|=\|\underline{t}\|\|\hat{\Pi}(\mathbf{i})-\hat{\Pi}(\mathbf{j})\||\cos(\theta(\underline{t},\mathbf{i},\mathbf{j}))|>\|\hat{\Pi}(\mathbf{i})-\hat{\Pi}(\mathbf{j})\|H^{-n}2^{bn}>H^{-n},
$$
which implies the claim. To finish the proof, we need to verify the dimension estimate.

For every $(\mathbf{i},\mathbf{j})\in\mathcal{A}_n$, the set $\bigg\{\underline{t}: 1/m\leq\|\underline{t}\|\leq M~\&~|\cos(\theta(\underline{t},\mathbf{i},\mathbf{j}))| \leq mH^{-n}2^{bn}\bigg\}$ can be covered by $C(H^n2^{-bn}/m)^{N-1}$-many balls of radius $ MmH^{-n}2^{bn}$, where $C>0$ is a universal constant independent of $n,m,M\in\N$, and $(\mathbf{i},\mathbf{j})\in\mathcal{A}_n$. Thus, we have for every $\tilde{N}\geq1$
\begin{align*}
    \mathcal{H}_{ MmH^{-\tilde{N}}2^{-\tilde{N}}}^{s}(\mathbf{F}_{m,M,H})&\leq \sum_{n=\tilde{N}}^{\infty}CL^{2n}(H^n2^{-bn}/m)^{N-1}\bigg(MmH^{-n}2^{bn}\bigg)^{s}\\&=\sum_{n=\tilde{N}}^{\infty}C''\left(L^2(H2^{-b})^{(N-1-s)}\right)^n<\infty
\end{align*}
provided $L^2(H2^{-b})^{N-1-s}<1$, where $C''=CM^sm^{s-N+1}$. This implies that $\dim_{H}\mathbf{F}_{m,M,H}\leq \frac{2\log L}{\log H-b\log2}+N-1$. Hence, we have $\dim_{H}\bigcap_{H=2^b}^\infty\mathbf{F}_{m,M,H}\leq N-1$, and so, $\dim_H\mathbf{F}\leq N-1$.
\end{proof}

Let us note that one could modify the definition of the exceptional set $\mathbf{F}$ slightly and show with some modification of the argument of Proposition~\ref{ESC for gerenalsystem} that \eqref{eq:ESCCFS} holds through a subsequence of integers $n\in\N$, and the dimension estimate remains valid by replacing the Hausdorff dimension with the packing dimension, following the lines of the proof of Hochman \cite{Hochman2022}.

\section{Calculation of \texorpdfstring{$h_{RW}$}{RW entropy} for general CFS}\label{sec:rw}
Recall the self-similar IFS
\begin{equation}\label{eq:goal2}
    \mathcal{S}_{\underline{\lambda},\underline{t}}=\{f_{i,j}=\lambda_{i,j}x+t_{i}(1-\lambda_{i,j}):~(i,j)\in I\}
\end{equation}
with a common fixed point structure. The next theorem will provide the formula for the Hausdorff dimension of self-similar measures corresponding to the IFS $\mathcal{S}_{\underline{\lambda},\underline{t}}$ by calculating the random walk entropy.

\begin{theorem}\label{thm:RWentropy} Let $\mathcal{S}_{\underline{\lambda},\underline{t}}$ be a self-similar IFS defined in \eqref{eq:goal2}. Suppose that $\mathcal{S}_{\underline{\lambda},\underline{t}}$ satisfies the ESC for CFS. Let $\mathbf{p}=(p_{i,j})_{(i,j)\in I}$ be a probability vector, and let $\mu$ be a self-similar measure corresponding to the probabilistic self-similar IFS $(\mathcal{S}_{\underline{\lambda},\underline{t}},\mathbf{p})$. Then, 
\begin{multline}\label{eq:RWform}
h_{RW}(\mu)=-\sum_{(i,j)\in I}p_{i,j}\log(p_{i,j})\\
+\sum_{(i,j)\in I }p_{i,j}\sum_{k=1}^{\infty} \sum_{l=0}^k\binom{k}{l}p_{i,j}^l\left(\sum_{\substack{m=1\\m\neq j}}^{n_i}p_{i,m}\right)^{k-l}\left(1-\sum_{m=1}^{n_i}p_{i,m}\right)\log\left(\dfrac{l+1}{k+1}\right).
\end{multline}
\end{theorem}
\begin{proof}
Let us first recall that
$$
H_{n}=-\sum_{\mathbf{i}\in \Sigma_{n}}p_{\mathbf{i}}\log\bigg(\sum_{\substack{\mathbf{j}\in \Sigma_{n}\\ f_{\mathbf{i}}\equiv f_{\mathbf{j}}}}p_\mathbf{j}\bigg)\text{ and }h_{RW}(\mu)=\lim_{n\to\infty}\frac{H_n}{n}.
$$
For every $n\in\N$ and $\mathbf{i},\mathbf{j}\in\Sigma_n$, if $\lambda_{\mathbf{i}}\neq\lambda_{\mathbf{j}}$ then $f_{\mathbf{i}}\not\equiv f_{\mathbf{j}}$ clearly, and if $\lambda_{\mathbf{i}}=\lambda_{\mathbf{j}}$ then by ESC for CFS either $\mathbf{i}$ and $\mathbf{j}$ have the block structure, and in particular, $f_{\mathbf{i}}\equiv f_{\mathbf{j}}$ or $f_{\mathbf{i}}(0)\neq f_{\mathbf{j}}(0)$. Let 
$$
\mathcal{C}(\mathbf{i})=\{\mathbf{j}\in\Sigma_{|\mathbf{i}|}:\text{ $\mathbf{i}$ and $\mathbf{j}$ have the same block structure}\}.$$
Then
$$
H_n=-\sum_{\mathbf{i}\in \Sigma_{n}}p_{\mathbf{i}}\log\bigg(\sum_{\mathbf{j}\in \mathcal{C}(\mathbf{i})}p_\mathbf{j}\bigg).
$$

It is enough to prove that the sequence $H_{n+1}-H_{n}$ converges to the right-hand side of \eqref{eq:RWform}. Indeed, if $\lim_{n\to\infty}(H_{n+1}-H_{n})=h$ then for each $\epsilon>0$ there exists a $N_0\in \mathbb{N}$ such that $h-\epsilon <H_{n+1}-H_n< h+\epsilon$ for all $n>N_{0}$. Thus, for all $k\in\N$ we have 
\begin{align*}
    H_{N_0+k}-H_{N_0}&=\sum_{p=1}^{k}(H_{N_0+p}-H_{N_0+p-1})<k(h+\epsilon),\text{ and similarly}\\
    H_{N_0+k}-H_{N_0}&>k(h-\epsilon).
    \end{align*}
    Thus, by squeeze theorem, we have $$h-\epsilon\leq \lim_{k\to \infty}\frac{H_{N_0+k}}{k+N_0}=h_{RW}(\mu)\leq h+\epsilon.$$
    Since $\epsilon>0$ was arbitrary, we get that $h=h_{RW}(\mu)$ and the claim of the theorem follows. 
    
Now, we will prove that the sequence $H_{n+1}-H_{n}$ converges. First, let us observe the following simple fact: for every $\mathbf{i}\in\Sigma^*$
$$
\sum_{\mathbf{j}\in\mathcal{C}(\mathbf{i})}p_{\mathbf{j}}=p_{\mathbf{i}}\prod_{k=1}^{B^{\mathbf{i}}}\dfrac{|b_k^{\mathbf{i}}|!}{\prod_{(i,j)\in I}(\#_{i,j}b_k^{\mathbf{i}})!}.
$$
Hence, for every $(i,j)\in I$
\begin{equation}\label{eq:sumC}
\frac{\sum_{\mathbf{j}'\in\mathcal{C}((i,j)\mathbf{i})} p_{\mathbf{j}'}}{\sum_{\mathbf{j}\in\mathcal{C}(\mathbf{i})} p_{\mathbf{j}}}=\begin{cases}
    p_{i,j} & \text{ if }|b_1^{(i,j)\mathbf{i}}|=1,\\[2pt]
    p_{i,j}\dfrac{|b_1^{\mathbf{i}}|+1}{\#_{i,j}b_1^{\mathbf{i}}+1} & \text{ if }|b_1^{(i,j)\mathbf{i}}|=1+|b_1^{\mathbf{i}}|.\\
\end{cases}
\end{equation}
\begin{align*}
    H_{n+1}-H_{n}&=-\sum_{\mathbf{i}\in \Sigma_{n+1}}p_{\mathbf{i}}\log\bigg(\sum_{\mathbf{j}\in\mathcal{C}(\mathbf{i})}p_\mathbf{j}\bigg)+\sum_{\mathbf{i}\in \Sigma_{n}}p_{\mathbf{i}}\log\bigg(\sum_{\mathbf{j}\in\mathcal{C}(\mathbf{i})}p_\mathbf{j}\bigg)\\
    &=-\sum_{\mathbf{i}\in \Sigma_{n}}\sum_{(i,j)\in I }p_{i,j}p_{\mathbf{i}}\log\bigg(\frac{\sum_{\mathbf{j}'\in\mathcal{C}((i,j)\mathbf{i})} p_{\mathbf{j}'}}{\sum_{\mathbf{j}\in\mathcal{C}(\mathbf{i})} p_{\mathbf{j}}}\bigg)
    \intertext{by \eqref{eq:sumC}}
    &=-\sum_{(i,j)\in I }\sum_{\substack{\mathbf{i}\in \Sigma_{n}\\ |b_1^{(i,j)\mathbf{i}}|=1}}p_{i,j}p_{\mathbf{i}}\log p_{i,j}-\sum_{(i,j)\in I }\sum_{\substack{\mathbf{i}\in \Sigma_{n}\\ |b_1^{(i,j)\mathbf{i}}|=|b_1^{\mathbf{i}}|+1}}p_{i,j}p_{\mathbf{i}}\log\left(p_{i,j}\dfrac{|b_1^{\mathbf{i}}|+1}{\#_{i,j}b_1^{\mathbf{i}}+1}\right)\\
    &=-\sum_{(i,j)\in I}p_{i,j}\log p_{i,j}-\sum_{(i,j)\in I }\sum_{\substack{\mathbf{i}\in \Sigma_{n}\\ |b_1^{(i,j)\mathbf{i}}|=|b_1^{\mathbf{i}}|+1}}p_{i,j}p_{\mathbf{i}}\log\left(\dfrac{|b_1^{\mathbf{i}}|+1}{\#_{i,j}b_1^{\mathbf{i}}+1}\right)\\
    &=-\sum_{(i,j)\in I}p_{i,j}\log p_{i,j}\\
    &\quad+\sum_{(i,j)\in I }p_{i,j}\sum_{k=1}^{n-1} \sum_{l=0}^k\binom{k}{l}p_{i,j}^l\left(\sum\limits_{\substack{n=1\\m\neq j}}^{n_i}p_{i,m}\right)^{k-l}\left(1-\sum_{m=1}^{n_i}p_{i,m}\right)\log\left(\dfrac{l+1}{k+1}\right)\\
    &\quad+\sum_{(i,j)\in I } p_{i,j}\sum_{l=0}^n\binom{n}{l}p_{i,j}^l\left(\sum_{\substack{m=1\\m\neq j}}^{n_i}p_{i,m}\right)^{n-l}\log\left(\dfrac{l+1}{n+1}\right).
\end{align*}
For the sequence
$$
a_n(i,j):=\sum_{l=0}^n\binom{n}{l}p_{i,j}^l\left(\sum_{\substack{m=1\\m\neq j}}^{n_i}p_{i,m}\right)^{n-l}\log\left(\dfrac{l+1}{n+1}\right),
$$
one can easily see that
$$
0\geq a_n(i,j)\geq -\log(n+1)\left(\sum_{m=1}^{n_i}p_{i,m}\right)^n.
$$
Thus, $a_n(i,j)\to0$ as $n\to\infty$, moreover, the series $\sum_{n=1}^\infty|a_n(i,j)|$ is convergent. Hence,
\begin{align*}
\lim_{n\to\infty}H_{n+1}-H_n=&-\sum_{(i,j)\in I}p_{i,j}\log p_{i,j}\\&+\sum_{(i,j)\in I }p_{i,j}\sum_{k=1}^{\infty} \sum_{l=0}^k\binom{k}{l}p_{i,j}^l\left(\sum_{\substack{m=1\\m\neq j}}^{n_i}p_{i,m}\right)^{k-l}\left(1-\sum_{m=1}^{n_i}p_{i,m}\right)\log\left(\dfrac{l+1}{k+1}\right),
\end{align*}
which had to be shown.
\end{proof}

\begin{proof}[Proof of Theorem~\ref{thm:main}]
    The claim follows by combining Theorem~\ref{MRandom}, Proposition~\ref{ESC for gerenalsystem} and Theorem~\ref{thm:RWentropy}.
\end{proof}

\section{Hausdorff dimension of the attractor corresponding to general CFS}\label{resultforattrac}

In this section, we will prove Theorem \ref{dimattractor}, which is our other main result. First, let us consider the following technical lemma:

\begin{lemma}\label{lem:tehc}
    Let $x_1,\ldots,x_n\in\R$ and let $A$ be an $n\times n$ matrix such that 
    $$
    A_{i,j}=\begin{cases}
        -1 & \text{if $i=j$,}\\
        x_j-1 & \text{otherwise.}
    \end{cases}
    $$
    Then
    $$
    \det(A)=(n-1)(-1)^{n+1}\prod_{k=1}^nx_k+(-1)^{n}\sum_{k=1}^n\prod_{\substack{\ell=1\\\ell\neq k}}^nx_\ell.
    $$
\end{lemma}

\begin{proof}
    Let us argue by induction for $n$. For $n=2$,
    $$
    \det\begin{pmatrix} -1 & x_2-1 \\ x_1-1 & -1\end{pmatrix}=(-1)^2-(x_2-1)(x_1-1)=x_1+x_2-x_1x_2,
    $$
    hence, the claim holds. Now, suppose that the claim holds for $n$. Then, by a simple row manipulation
    \begin{align*}
    &\det\begin{pmatrix}
        -1 & x_2-1 & \cdots & x_{n+1}-1\\
        x_1-1 & \ddots & & \vdots\\
        \vdots &  & \ddots & x_{n+1}-1\\
        x_1-1 & \cdots & x_n-1 & -1
    \end{pmatrix}=\det\begin{pmatrix}
        -x_1 & x_2 & 0& \cdots & 0\\
        x_1-1 & -1 & x_3-1 & \cdots & x_{n+1}-1\\
        & & \ddots &   & \vdots\\
        \vdots &  &  & \ddots &  x_{n+1}-1\\
        x_1-1 & \cdots & & x_n-1 & -1
    \end{pmatrix}
    \intertext{by the inductive assumption}
    &\quad=-x_1\left((n-1)(-1)^{n+1}\prod_{k=2}^{n+1}x_k+(-1)^n\sum_{k=2}^{n+1}\prod_{\substack{\ell=2\\\ell\neq k}}^{n+1}x_\ell\right)\\
    &\hspace{7cm}-x_2\det\begin{pmatrix}         x_1-1 & x_3-1 & \cdots & x_{n+1}-1\\         x_1-1 & -1 & & \vdots\\         \vdots & \ddots & \ddots & x_{n+1}-1\\         x_1-1 & \cdots & x_n-1 & -1     \end{pmatrix}
    \intertext{and again by a simple row manipulation}
    &\quad=-x_1\left((n-1)(-1)^{n+1}\prod_{k=2}^{n+1}x_k+(-1)^n\sum_{k=2}^{n+1}\prod_{\substack{\ell=2\\\ell\neq k}}^{n+1}x_\ell\right)\\
    &\hspace{7cm}-x_2\det\begin{pmatrix}         x_1-1 & x_3-1 & \cdots & x_{n+1}-1\\         0 & -x_3 & 0 & 0\\         \vdots & \ddots & \ddots & 0\\         0 & \cdots & 0 & -x_{n+1}     \end{pmatrix}\\
     &\quad=-x_1\left((n-1)(-1)^{n+1}\prod_{k=2}^{n+1}x_k+(-1)^n\sum_{k=2}^{n+1}\prod_{\substack{\ell=2\\\ell\neq k}}^{n+1}x_\ell\right)-x_2(-1)^{n-1}(x_1-1)\prod_{k=3}^{n+1}x_k\\
    &\quad=n(-1)^{n+2}\prod_{k=1}^{n+1}x_k+(-1)^{n+1}\sum_{k=2}^{n+1}\prod_{\substack{\ell=1\\\ell\neq k}}^{n+1}x_\ell+(-1)^{n+1}\prod_{\ell=2}^{n+1}x_\ell,
    \end{align*}
which completes the proof.
\end{proof}

Let $\mathcal{S}_{\underline{\lambda},\underline{t}}$ be the self-similar IFS with common fixed point structure defined in \eqref{eq:mainIFS}
$$
\mathcal{S}_{\underline{\lambda},\underline{t}}=\{f_{i,j}(x)=\lambda_{i,j}x+t_i(1-\lambda_{i,j}):~1\leq i\leq N~\&~1\leq j\leq n_i\}.
$$
We assume that $\mathcal{S}_{\underline{\lambda},\underline{t}}$ satisfies the ESC for CFS. For every $i=1,\ldots,N$ and every $n\in\N$, let 
$$
T_i^{(n)}:=\{(i,j_1)\cdots(i,j_n):1\leq j_1\leq\cdots\leq j_n\leq n_i\}.
$$
and let $\mathcal{V}_n=\bigcup_{i=1}^N\bigcup_{m=1}^nT_i^{(m)}$. Clearly, $\mathbf{i}=b_1^{\mathbf{i}}$ for every $\mathbf{i}\in\mathcal{V}_n$. Let us define a directed graph $\mathcal{G}_n=(\mathcal{V}_n,\mathcal{E}_n)$ with vertices $\mathcal{V}_n$ and edges  
$$
\mathcal{E}_n=\{\mathbf{i}\to\mathbf{j}:~t_{\mathbf{i}}\neq t_{\mathbf{j}}\}.
$$
The graph $\mathcal{G}_n$ is strongly connected, that is, there exists a path (of length at most two) from $\mathbf{i}$ to $\mathbf{j}$ for any $\mathbf{i},\mathbf{j}\in\mathcal{V}_n$. Furthermore, for $N\geq3$ the graph $\mathcal{G}_n$ is aperiodic. 

For every $n\in\N$, we define a graph-directed self-similar IFS $\mathcal{S}_n=\{f_{\mathbf{j}}:(\mathbf{i}\to\mathbf{j})\in\mathcal{E}_n\}$. By Mauldin and Williams \cite[Theorem~1.1]{MauldinWilliams}, for every $\mathbf{i}\in\mathcal{V}_n$ there exists a unique non-empty compact set $A_n^{\mathbf{i}}$ such that
$$
A_n^{\mathbf{i}}=\bigcup_{\mathbf{j}:~\mathbf{i}\to\mathbf{j}\in\mathcal{E}_n}f_{\mathbf{j}}(A_n^{\mathbf{j}}).
$$
Clearly, $A_n:=\bigcup_{\mathbf{i}\in\mathcal{V}_n}A_n^{\mathbf{i}}\subseteq A$ for every $n\in\N$, where $A$ is the attractor of $\mathcal{S}$.

We say that the graph-directed self-similar IFS $\mathcal{S}_n$ satisfies the \texttt{exponential separation condition for graph directed system (ESC for GDS)} if there exists $b>0$ such that for every $m\in\N$ and every two different paths $\mathbf{i}_1\to\cdots\to\mathbf{i}_m$ and $\mathbf{j}_1\to\cdots\to\mathbf{j}_m$ of length $m$ in $\mathcal{G}_n$ with $\lambda_{\mathbf{i}_2\cdots\mathbf{i}_m}=\lambda_{\mathbf{j}_2\cdots\mathbf{j}_m}$ and same initial vertices $\mathbf{i}_1=\mathbf{j}_1$, we have $|f_{\mathbf{i}_2\cdots\mathbf{i}_m}(0)-f_{\mathbf{j}_2\cdots\mathbf{j}_m}(0)|\geq 2^{-bm}$.

\begin{lemma}\label{lem:ESCGD}
    Let $\underline{\lambda}\notin G=\bigcup_{n=1}^\infty G_{1/n}$, where $G_{1/n}$ is defined in \eqref{eq:G1}. If the IFS $\mathcal{S}_{\underline{\lambda},\underline{t}}$ satisfies the ESC for CFS, then the graph-directed self-similar IFS $\mathcal{S}_n$ satisfies the ESC for GDS.
\end{lemma}

\begin{proof}
    Let $\mathbf{i}_1\to\cdots\to\mathbf{i}_m$ and $\mathbf{j}_1\to\cdots\to\mathbf{j}_m$ be different allowed paths with $\mathbf{i}_1=\mathbf{j}_1$ and $\lambda_{\mathbf{i}_2\cdots\mathbf{i}_m}=\lambda_{\mathbf{j}_2\cdots\mathbf{j}_m}$. Since $\lambda_{\mathbf{i}_2\cdots\mathbf{i}_m}=\lambda_{\mathbf{j}_2\cdots\mathbf{j}_m}$, we have $\#_{i,j}\mathbf{i}_2\cdots\mathbf{i}_m=\#_{i,j}\mathbf{j}_2\cdots\mathbf{j}_m$ for every $(i,j)\in I$, and in particular, $|\mathbf{i}_2\cdots\mathbf{i}_m|=|\mathbf{j}_2\cdots\mathbf{j}_m|$. Furthermore, since the paths are different $f_{\mathbf{i}_2\cdots\mathbf{i}_m}\not\equiv f_{\mathbf{j}_2\cdots\mathbf{j}_m}$ and the IFS $\mathcal{S}$ satisfies the ESC for CFS
    $$
    |f_{\mathbf{i}_2\cdots\mathbf{i}_m}(0)-f_{\mathbf{j}_2\cdots\mathbf{j}_m}(0)|>2^{-b|\mathbf{i}_2\cdots\mathbf{i}_m|}\geq 2^{-bnm},
    $$
    which had to be proven.
\end{proof}

We remark that Prokaj and Simon \cite[Corollary 7.2]{ProkajSimon} have already studied the dimension of the attractor of graph-directed IFSs under a stronger version of ESC, namely, when the IFS formed by the union of all the maps in the graph-directed system satisfies the ESC. However, the common fixed point structure of the original IFS from which our graph-directed system was deduced stands in the way of this condition. So, we cannot apply their result directly.

Now, let us define the following matrix with elements indexed by $\mathcal{V}_n$: 
$$
(B_n^{(s)})_{\mathbf{i},\mathbf{j}}:=\begin{cases}
    0 & \text{if }t_{\mathbf{i}}=t_{\mathbf{j}},\\
    \lambda_{\mathbf{j}}^s & \text{if }t_{\mathbf{i}}\neq t_{\mathbf{j}}.
\end{cases}
$$
Since $\mathcal{G}_n$ is strongly connected, we get that $B_n^{(s)}$ is an irreducible matrix with non-negative entries. Moreover, the matrix $B_n^{(s)}$ is primitive for each $n\in \mathbb{N}.$

{ \begin{lemma}\label{Eqspect}  Let $\underline{\lambda}\notin G=\bigcup_{n=1}^\infty G_{1/n}$, where $G_{1/n}$ is defined in \eqref{eq:G1}. If the IFS $\mathcal{S}_{\underline{\lambda},\underline{t}}$ satisfies the ESC for CFS, then for each $n\in \mathbb{N}$, $$
\dim_HA_n=\min\{1,s_n\},
$$
where $s_n$ is the unique solution of the equation $\rho(B_n^{(s_n)})=1$, where $\rho(B)$ denotes the spectral radius of the matrix $B$.
\end{lemma}}
{
\begin{proof} It is well known that $
\dim_H A_{n}\leq \min\{1,s_n\}.$ Next, we will show the lower bound. 
 For each $\bold{i}\in \mathcal{V}_n$ and $m\in \mathbb{N},$ we define a set a follows
 $$C_{\bold{i}}^{m}:=\{\bold{i}_{1}\to \bold{i}_{2} \to \cdots\to\mathbf{i}_{m}: \bold{i}_{1}=\mathbf{i}_{m}=\bold{i}\}.$$
 The set $C_{\bold{i}}^{m}$ is the collection of all $m$ length paths that start and end at the vertex $\bold{i}.$
Now, we define a self-similar IFS $\mathcal{F}_{n}^{m}$ as follows
$$\mathcal{F}_{n}^{m}:=\{f_{\bold{i}_{2}}\circ f_{\bold{i}_{3}}\circ \dots \circ f_{\bold{i}_{m-1}}\circ f_{\bold{i}}: \bold{i}\to \bold{i}_{2} \to \cdots\to\mathbf{i}_{m-1}\to \bold{i}\in C_{\bold{i}}^{m}\}.$$
Let $A_{n}^{m}$ be the attractor of the IFS $\mathcal{F}_{n}^{m}$. Clearly, $A_{n}^{m}\subseteq A_{n}$ for each $m\in \mathbb{N}.$ By Lemma \ref{lem:ESCGD}, it is clear that the IFS $\mathcal{F}_{n}^{m}$ satisfies ESC. Then by Theorem \ref{MainHochmanresult},
$$\dim_{H}(A_{n}^{m})=\min \{1,s^{m}\},$$
where $s^{m}$ is the similarity dimension of the self-similar IFS $\mathcal{F}_{n}^{m}$. Note that the sequence $\{s^{m}\}_{m=1}^{\infty}$ is uniformly bounded above by the similarity dimension of the IFS $\mathcal{S}_{\underline{\lambda},\underline{t}}$. For a fixed $m\in \mathbb{N},$ the sequence $\{s^{km}\}_{k=1}^{\infty}$ is monotonically increasing sequence. Let $\lim_{k\to \infty} s^{km}=s^{*m}$. Hence, 
\begin{equation}\label{eq:Anlower}
\min\{1,s^{*m}\}\leq\dim_HA_n\text{ for each }m\in\N.
\end{equation}
The following key fact is due to the Perron-Frobenius theorem,
$$\lim_{m\to \infty} ((B_{n}^{(s)})^{m}_{\bold{i}, \bold{i}})^{1/m}=\rho(B_n^{(s)}).$$
This implies that 
\begin{equation}\label{Peronresult}
    \lim_{m\to \infty} \bigg(\sum_{\bold{j}\in \mathcal{F}_{n}^{m}}\lambda_{\bold{j}}^{s}\bigg)^{1/m}=\rho(B_n^{(s)}).
\end{equation}
For each fixed $M\in \mathbb{N},$
$$\lim_{k\to \infty} \bigg(\sum_{\bold{j}\in \mathcal{F}_{n}^{Mm}}\lambda_{\bold{j}}^{s^{km}}\bigg)=\sum_{\bold{j}\in \mathcal{F}_{n}^{Mm}}\lambda_{\bold{j}}^{s^{*m}}\leq 1.$$ 
By \eqref{Peronresult}, there exists an $N\in \mathbb{N}$ such that 
$$\left|\left(\sum_{\bold{j}\in \mathcal{F}_{n}^{km}}\lambda_{\bold{j}}^{s^{*m}}\right)^{1/km}-\rho(B_n^{(s^{*m})})\right|\leq \epsilon \quad \text{for all}\quad  k\geq N.$$
This implies that $\rho(B_n^{(s^{*m})})\leq 1+\epsilon$. Since   $\epsilon$ is arbitrary, we have $\rho(B_n^{(s^{*m})})\leq 1.$ It is well known that $s\mapsto\rho(B_n^{(s)})$ is a decreasing function of $s$. Thus, we get $s_n\leq s^{*m}.$ By \eqref{eq:Anlower}, we get 
 $$\min\{1,s_{n}\}\leq \min\{1,s^{*m}\}\leq \dim_H A_{n}.$$
\end{proof}}

By Lemma~\ref{Eqspect} and $A_{n}\subseteq A$, we have 
\begin{equation}\label{eq:lowerb}
\min\{1,\lim_{n\to\infty}s_n\}\leq\dim_HA.
\end{equation}
Note that $s_n$ is monotone increasing and bounded by the similarity dimension of $\mathcal{S}$, so the limit $\lim_{n\to\infty}s_n$ exists.

\begin{lemma}\label{lem:limit}
    Let $s_0$ be the unique solution of the equation \eqref{eq:dimatt}. Then $\lim\limits_{n\to\infty}s_n=s_0$.
\end{lemma}

\begin{proof}
 Let us define the following $N\times N$ matrix $C_n^{(s)}$ as follows:
$$
(C_n^{(s)})_{i,k}=\begin{cases}
    \sum_{\mathbf{j}\in \cup_{m=1}^{n} T_{k}^{(m)}}\lambda_{\mathbf{j}}^s & \text{if }i\neq k,\\
    0 & \text{if }i=k.
\end{cases}
$$
First, we show that $\rho(B_n^{(s)})=\rho(C_n^{(s)})$. By Perron-Frobenius theorem, there exists $v=(v_1,\ldots,v_N)^T$ with strictly positive entries such that $C_n^{(s)}v=\rho(C_n^{(s)})v$. Now, we define the vector $w=(w_{\mathbf{i}})_{\mathbf{i}\in\mathcal{V}_n}$ as follows: $w_{\mathbf{i}}=v_i$ if $t_{\mathbf{i}}=t_i$. Then for every $\mathbf{i}\in\mathcal{V}_n$ with $t_{\mathbf{i}}=t_i$
$$
(B_n^{(s)}w)_{\mathbf{i}}=\sum_{\mathbf{j}:~\mathbf{i}\to\mathbf{j}\in\mathcal{E}_n}\lambda_{\mathbf{j}}^sw_{\mathbf{j}}=\sum_{\substack{k=1\\k\neq i}}^N\left(\sum_{\mathbf{j}\in \cup_{m=1}^{n}T_k^{(m)}}\lambda_{\mathbf{j}}^s\right)v_{k}=\rho(C_n^{(s)})v_i=\rho(C_n^{(s)})w_{\mathbf{i}}.
$$
Again applying Perron-Frobenius Theorem, we get that $\rho(B_n^{(s)})=\rho(C_n^{(s)})$. Thus, $\rho(C_n^{(s_n)})=1$, and in particular, $\det(C_n^{(s_n)}-\mathrm{Id})=0$, where $\mathrm{Id}$ denotes the identity matrix. By denoting 
$$
a_n^i(s):=\sum_{\mathbf{j}\in \cup_{m=1}^{n} T_i^{(m)}}\lambda_{\mathbf{j}}^s+1~\text{for}~i\in \{1,2,\dots,N\},
$$
we get by Lemma~\ref{lem:tehc}
$$
0=|\det(C_n^{(s_n)}-\mathrm{Id})|=(N-1)\prod_{i=1}^Na_n^i(s_n)-\sum_{k=1}^N\prod_{\substack{\ell=1\\\ell\neq k}}^Na_n^\ell(s_n).
$$
It is easy to see that for every $s\in(0,\infty)$, $\lim_{n\to\infty}a_n^i(s)=\prod_{j=1}^{n_i}(1-\lambda_{i,j}^s)^{-1}$, and by continuity, $s_0$ satisfies the equation
$$
0=(N-1)\prod_{i=1}^N\prod_{j=1}^{n_i}(1-\lambda_{i,j}^{s_0})^{-1}-\sum_{k=1}^N\prod_{\substack{\ell=1\\\ell\neq k}}^N\prod_{j=1}^{n_\ell}(1-\lambda_{\ell,j}^{s_0})^{-1},
$$
and so
$$
0=(N-1)-\sum_{k=1}^N\prod_{j=1}^{n_k}(1-\lambda_{k, j}^{s_0}),
$$
which had to be proven.
\end{proof}

To finish the proof, we show that $s_0$ is an upper bound. Let $T_i^{(\infty)}=\bigcup_{n=1}^\infty T_i^{(n)}$, and for $s>0$ let 
$$
(C_\infty^{(s)})_{i,k}=\begin{cases}
    \sum_{\mathbf{j}\in T_{k}^{(\infty)}}\lambda_{\mathbf{j}}^s=\prod_{j=1}^{n_k}(1-\lambda_{k,j}^s)^{-1}-1 & \text{if }i\neq k,\\
    0 & \text{if }i=k.
\end{cases}
$$
Since $C_n^{(s)}\to C_{\infty}^{(s)}$ as $n\to\infty$, the quantity $s_0$ is the unique solution of the equation $\rho(C_\infty^{(s_0)})=1$.

\begin{lemma}\label{lem:ub}
    Let $\mathcal{S}$ be the self-similar IFS defined in \eqref{eq:mainIFS}. Then for every $\underline{\lambda}=(\lambda_{i,j})_{i=1,j=1}^{N,n_i}$ and $\underline{t}=(t_i)_{i=1}^N$
    $$\dim_{H}\Lambda\leq \min\{1,s_0\},$$
    where $\Lambda$ is the attractor of $\mathcal{S}$.
\end{lemma}

\begin{proof}
First, observe that the set $\Gamma=\Pi(\{\mathbf{i}\in\Sigma: B^{\mathbf{i}}<\infty\})$ is countable, and so, it is enough to cover the set $\Lambda\setminus\Gamma$. Without loss of generality, we may assume that $t_1<\cdots<t_N$. Observe that for every $n\in\N$, the set $\Lambda\setminus\Gamma$ can be covered by 
$$
\left\{f_{\mathbf{i}_1\cdots\mathbf{i}_n}([t_1,t_N]):\mathbf{i}_k\in T_{i_k}^{(\infty)}\text{ and }i_k\neq i_{k+1}\text{ for every }k=1,\ldots,n-1\right\},
$$
with diameters at most $\lambda_{\max}^n|t_N-t_1|$. Now, let $s>s_0$ be arbitrary, and so $\rho(C_\infty^{(s)})<1$. Denote $\|.\|_1$ the $1$-norm for $N\times N$ matrices, and let $\ell\geq1$ be such that $\left\|(C_{\infty}^{(s)})^m\right\|_1\leq\left(\frac{\rho(C_\infty^{(s)})+1}{2}\right)^m$ for every $m\geq \ell$. Hence,
\begin{align*}
    \mathcal{H}^s_{\lambda_{\max}^n|t_N-t_1|}(\Lambda\setminus\Gamma)&\leq\sum_{\substack{i_1,\ldots,i_n\in\{1,\ldots,N\}\\i_{k}\neq i_{k+1}\text{ for }k=1,\ldots,n-1}}\sum_{\mathbf{i}_1\in T_{i_1}^{(\infty)}}\cdots\sum_{\mathbf{i}_n\in T_{i_n}^{(\infty)}}\lambda_{\mathbf{i}_1}^s\cdots\lambda_{\mathbf{i}_n}^s|t_N-t_1|^{s}\\
    &\leq\|(C_\infty^{(s)})^n\|_1|t_N-t_1|^s\leq \left(\frac{\rho(C_\infty^{(s)})+1}{2}\right)^n|t_N-t_1|^s
\end{align*}
for every $n\geq \ell$. Hence, $\mathcal{H}^s(\Lambda\setminus \Gamma)=0$ for $s>s_0$, which implies the claim.
\end{proof}

\begin{proof}[Proof of Theorem~\ref{dimattractor}]
    The claim follows by combining \eqref{eq:lowerb}, Lemma~\ref{lem:limit} and Lemma~\ref{lem:ub}.
\end{proof}

\bibliographystyle{amsplain}
\bibliography{cikk}

\end{document}